\documentclass{amsart}

\usepackage{amsmath,amssymb,amsthm,comment,mathrsfs}
\usepackage[usenames]{color}
\usepackage{tikz,tikz-cd}
\usetikzlibrary{matrix,arrows,decorations.pathmorphing}
\usepackage[usenames]{color}
\usepackage[margin=1.0in]{geometry}
\usepackage{hyperref}
\hypersetup{pdfstartview={XYZ null null 1.00}, pdfpagemode=UseNone, colorlinks,breaklinks, linkcolor=blue,urlcolor=blue, anchorcolor=blue,citecolor=blue}

\usepackage[alphabetic,short-journals,short-publishers]{amsrefs}
\usepackage{enumitem}
\usepackage{newclude}

\newcommand{\frakgl}{\mathfrak{gl}}
\newcommand{\Q}{\mathbb{Q}}
\newcommand{\Z}{\mathbb{Z}}
\newcommand{\R}{\mathbb{R}}
\newcommand{\C}{\mathbb{C}}
\newcommand{\F}{\mathbb{F}}
\newcommand{\A}{\mathbb{A}}
\newcommand{\T}{\mathbb{T}}
\newcommand{\Qbar}{\overline{\mathbb{Q}}}

\newcommand{\GL}{\mathrm{GL}}
\newcommand{\SL}{\mathrm{SL}}

\newcommand{\ad}{\mathrm{ad}}
\newcommand{\Hom}{\mathrm{Hom}}
\newcommand{\End}{\mathrm{End}}
\newcommand{\Gal}{\mathrm{Gal}}
\newcommand{\M}{\mathrm{M}}

\newcommand{\univ}{\mathrm{univ}}
\newcommand{\lra}{\longrightarrow}

\newcommand{\Frob}{\mathrm{Frob}}
\newcommand{\Nm}{\mathrm{Nm}}

\newcommand{\calO}{\mathcal{O}}

\newcommand{\calS}{\mathcal{S}}
\newcommand{\calL}{\mathcal{L}}

\newcommand{\calG}{\mathcal{G}}
\newcommand{\rhobar}{\overline{\rho}}

\newcommand{\frakm}{\mathfrak{m}}
\newcommand{\frakp}{\mathfrak{p}}
\DeclareMathOperator{\Spec}{Spec}

\newcommand{\depth}{\mathrm{depth}}
\newcommand{\tr}{\mathrm{tr}}
\newcommand{\loc}{\mathrm{loc}}

\newcommand{\cris}{\mathrm{cr}}

\newcommand{\st}{\mathrm{st}}
\newcommand{\fraka}{\mathfrak{a}}

\newcommand{\Ar}{\mathrm{Ar}}

\newcommand{\Art}{\mathrm{Art}}

\newcommand{\WD}{\mathrm{WD}}
\newcommand{\dR}{\mathrm{dR}}
\newcommand{\gr}{\mathrm{gr}}
\newcommand{\ab}{\mathrm{ab}}

\newcommand{\CNL}{\mathrm{CNL}}

\DeclareMathOperator{\Sym}{Sym}
\newcommand{\Fss}{F\mbox{-}\mathrm{ss}}
\newcommand{\rec}{\mathrm{rec}}
\newcommand{\recT}{\mathrm{rec}^T}
\newcommand{\Sp}{\mathrm{Sp}}

\newcommand{\Aut}{\mathrm{Aut}}
\newcommand{\Fil}{\mathrm{Fil}}
\newcommand{\frakz}{\mathfrak{z}}		
\newcommand{\Ext}{\mathrm{Ext}}
\newcommand{\red}{\mathrm{red}}

\newcommand{\rest}{|}
\newcommand{\rbar}{\overline{r}}

\newcommand{\bbT}{\mathbb{T}}
\DeclareMathOperator{\Res}{Res}
\DeclareMathOperator{\Lie}{Lie}
\DeclareMathOperator{\HT}{HT}

\newcommand{\tildeS}{\widetilde{S}}
\newcommand{\GO}{\mathrm{GO}}

\newcommand{\GSp}{\mathrm{GSp}}
\newcommand{\Liegl}{\mathfrak{gl}}
\newcommand{\transp}{\,{}^t\!}

\providecommand{\abs}[1]{\lvert #1 \rvert}

\newtheorem{mainthm}{Theorem}

\newtheorem{thm}[subsubsection]{Theorem}
\newtheorem{lem}[subsubsection]{Lemma}
\newtheorem{prop}[subsubsection]{Proposition}
\newtheorem{cor}[subsubsection]{Corollary}

\theoremstyle{definition}
\newtheorem{defn}[subsubsection]{Definition}

\theoremstyle{remark}
\newtheorem{rmk}[subsubsection]{Remark}
\newtheorem{eg}[subsubsection]{Example}

\theoremstyle{remark}

\setenumerate[0]{label=\arabic*.,ref=\arabic*}

\title[Deformations and adjoint Selmer groups]{Deformations of polarized automorphic Galois representations and adjoint Selmer groups}
\author{Patrick B. Allen}

\address{Department of Mathematics, University of Illinois at Urbana--Champaign,
Urbana, IL, USA}
\email{pballen@illinois.edu}

\DeclareRobustCommand{\SkipTocEntry}[5]{}

\begin{document}

\thanks{The author was partially supported by an AMS--Simons Travel Grant. Some of this work was completed while the author was a guest that the Max Planck Institute for Mathematics, and he thanks them for their hospitality.}

\begin{abstract}
We prove the vanishing of the geometric Bloch--Kato Selmer group for the adjoint representation of a Galois representation associated to regular algebraic polarized cuspidal automorphic representations under an assumption on the residual image. 
Using this, we deduce that the localization and completion of a certain universal deformation ring for the residual representation at the characteristic zero point induced from the automorphic representation is formally smooth of the correct dimension. 
We do this by employing the Taylor--Wiles--Kisin patching method together with Kisin's technique of analyzing the generic fibre of universal deformation rings. 
Along the way we give a characterization of smooth closed points on the generic fibre of Kisin's potentially semistable local deformation rings in terms of their Weil--Deligne representations.
\end{abstract}

\maketitle

\setcounter{tocdepth}{2}

\tableofcontents

\section*{Introduction}

Let $F$ be a number field, $S$ a finite set of finite places of $F$ containing all those above a fixed rational prime $p$, and let $F(S)$ be the maximal extension of $F$ unramified outside of $S$ and the Archimedean places. 
Given a $p$-adic representation $V$ of $\Gal(F(S)/F)$, Bloch and Kato \cite{BlochKato} defined certain subspaces
	\[ H_f^1(F(S)/F,V) \subseteq H_g^1(F(S)/F,V) \subseteq H^1(F(S)/F,V) \]
of the Galois cohomology group $H^1(F(S)/F,V)$, known as the \emph{Bloch--Kato Selmer group} and \emph{geometric Bloch--Kato Selmer group}, respectively. 
If $V$ is \mbox{de Rham}, resp. crystalline, then $H_g^1(F(S)/F,V)$, resp. $H_f^1(F(S)/F,V)$, is the subspace of $H^1(F(S)/F,V) = \Ext_{\Q_p[\Gal(F(S)/F)]}^1(\Q_p,V)$ of extensions of the trivial representation by $V$ that are \mbox{de Rham}, resp. crystalline.
When $V$ is absolutely irreducible and \mbox{de Rham}, the Fontaine--Mazur conjecture together with the philosophy of motives predict that $V$ should be the $p$-adic realization of some pure motive.
Inputting this into the conjectures of Beilinson--Bloch and Bloch--Kato, one obtains a conjectural relation between the dimension of $H_f^1(F(S)/F,V)$ and the order of vanishing of the $L$-function of the dual representation of $V$ at the point $s = 1$ \cite{FPR}*{II, \S3.4.5}.
One prediction of this conjecture is that if the representation $V$ is pure of motivic weight zero, then 
	\[ H_f^1(F(S)/F,V) = H_g^1(F(S)/F,V) = 0.\]
This is in accordance with a philosophy of Grothendieck that in a conjectural category of mixed motives, there should be no nontrivial extensions of pure motives of the same weight.

Let $E$ be a finite extension of $\Q_p$. Given any absolutely irreducible pure de Rham representation
	\[ \rho : \Gal(F(S)/F) \lra \GL_d(E), \]
one naturally obtains a pure weight zero representation called the adjoint representation: $\ad(\rho) = \mathfrak{gl}_d(E)$, the Lie algebra of $\GL_d(E)$, with $\Gal(F(S)/F)$-action given by composing $\rho$ with the adjoint action of $\GL_d(E)$. 
Then the Bloch--Kato conjecture predicts
	\[ H_g^1(F(S)/F,\ad(\rho)) = 0.\]
In the case where $\rho$ is the representation arising from an elliptic curve over $\Q$, this prediction was first proved by Flach \cite{FlachSymSquare} by a method using Euler systems, assuming that the elliptic curve has good reduction at $p$, that $p \ge 5$, and that the associated residual representation surjects onto $\GL_2(\F_p)$.
A corollary of the breakthrough work of Wiles and Taylor--Wiles is this vanishing in the case that $\rho$ is the representation coming certain modular forms.
This results from their so-called $R=\mathbb{T}$ theorem that equates a certain universal deformation ring of $\rhobar$ to a Hecke algebra. 
On way to think about this is that the tangent space of the deformation ring they consider at the characteristic zero point corresponding to the modular form is equal to its adjoint Bloch--Kato Selmer group, while the tangent space of the Hecke algebra at that point is trivial, since the Hecke algebra is reduced. 
Their work also had assumptions on the level of the modular form and on the residual representation. 
In \cite{KisinGeoDefs}, Kisin showed the vanishing of $H_g^1(G_{\Q,S},\ad(\rho))$ for modular forms of weight $k \ge 2$ and arbitrary level, assuming only a mild condition on the residual representation, and his proof uses some of the ideas of Taylor--Wiles.
We mention also the result of Weston \cite{WestonGeoEuler} which applies to non-CM forms with certain hypotheses on the level, but has no restriction on the residual representation. 
For general totally real, resp. CM fields, but still in dimension two, one can deduce results of this form from the $R[1/p] = \mathbb{T}[1/p]$ theorems \cite{KisinFinFlat}*{Theorem~3.4.11} and \cite{KW2}*{Propositions~9.2 and 9.3}, resp.  \cite{GeeKisinBM}*{Corollary~3.4.3}, whenever the assumptions of those theorems are satisfied. 

In higher dimensions, one is naturally led to consider the Galois representations associated to regular algebraic polarized cuspidal automorphic representations of $\GL_d$ over CM fields. 
These adjoint representations have a natural extension to a representation of the Galois group of the maximal totally real subfield, and this adjoint Selmer group has a natural interpretation as the tangent space of a polarized deformation ring. 
Although there has been great progress in modularity lifting in this context, almost all of the results prove only $R^{\red} = \mathbb{T}$ and do not imply vanishing of the adjoint Selmer groups, although 
some cases can be deduced using the $R = \bbT$ theorem of Clozel--Haris--Taylor \cite{CHT}*{Theorem~3.5.1}.  
Using a completely different method, namely 
the theory of eigenvarieties, Chenevier \cite{ChenevierFern}*{Theorem~F} proved that the vanishing of this adjoint Selmer group is equivalent to the vanishing of the $H^2$ of this adjoint representation under some local hypotheses. 
This is applicable, in particular, when the deformation problem is unobstructed (for example, see \cite{ChenevierFern}*{Appendix}).

The main observation used in this paper is that one can still use the Taylor--Wiles--Kisin patching method to deduce vanishing of the corresponding adjoint Selmer groups for automorphic Galois representations  provided one knows that the induced points on local deformation rings are smooth. 
Indeed, the method yields a ring $R_\infty$, an $R_\infty$-module $M_\infty$, and a control theorem that relates them to our deformation ring and a space of automorphic forms. 
The most subtle point in proving modularity lifting theorems is to understand the components of $R_\infty$ and how they relate to congruences between automorphic forms.
But if we are only interested in the infinitesimal deformation theory of the characteristic zero point coming from our automorphic Galois representation $\rho$, we can localize and complete at this point, and if we know that $\rho$ determines a smooth point on the local deformation rings, it also determines a smooth point on $R_\infty$. 
Then we can apply the Auslander--Buchsbaum formula to the completion and deduce that the localized and completed deformation ring acts freely on a finite dimensional vector space of cusp forms, from which we can deduce vanishing of the adjoint Selmer group. 
Before stating the main theorems, we set up some notation. 

Let $E$ be a finite extension of $\Q_p$ with ring of integers $\calO$ and residue field $\F$. 
Let $F$ be a CM field with maximal totally real subfield $F^+$. 
Let $\overline{F}$ be some fixed algebraic closure of $F$, and let $c \in \Gal(\overline{F}/F^+)$ be a choice of complex conjugation.
Let $S$ be a finite set of finite places of $F^+$ containing all places above $p$. 
Let $F(S)$ be the maximal extension of $F$ unramified outside of the places in $F$ above those in $S$. 
Note that $F(S)$ is Galois over $F^+$. 
Let
	\[ \rho : \Gal(F(S)/F) \lra \GL_d(E) \]
be a continuous, absolutely irreducible representation, and let $\ad(\rho)$ denote the Lie algebra $\mathfrak{gl}_d(E)$ of $\GL_d(E)$ with the adjoint action $\ad\circ \rho$ of $\Gal(F(S)/F)$. 

We assume there is a continuous totally odd character $\mu : G_{F^+} \rightarrow E^\times$ and an invertible symmetric matrix $P$ such that the pairing $\langle a, b \rangle = \transp a P^{-1} b$ on $E^d$ is perfect and satisfies 
	\[ \langle \rho(\sigma) a, \rho(c\sigma c) b \rangle = \mu(\sigma) \langle a, b \rangle, \]
for all $\sigma \in \Gal(F(S)/F)$. 
Since $\rho$ is absolutely irreducible, $P$ is unique up to scalar. 
We can then extend the action of $\Gal(F(S)/F)$ on $\ad(\rho)$ to an action of $\Gal(F(S)/F^+)$ by letting $c$ act by $X \mapsto -P\,{}^t\! X P^{-1}$, 
and this is independent of the choice of $c$ and of $P$. 
 
Finally, we recall that we can choose a $\Gal(F(S)/F)$-stable $\calO$-lattice in the representation space of $\rho$, so after conjugation we may assume that $\rho$ takes values in $\GL_d(\calO)$. The semisimplification of its reduction modulo the maximal ideal of $\calO$ does not depend on the choice of lattice, and we denote it by $\rhobar : \Gal(F(S)/F) \rightarrow \GL_d(\F)$.

\begin{mainthm}\label{thm:BKA}
Assume $p>2$. 
Assume there is a finite extension $L/F$ of CM fields, a regular algebraic polarizable cuspidal automorphic representation $\Pi$ of $\GL_d(\A_{L})$, and an isomorphism $\iota : \Qbar_p \xrightarrow{\sim}\C$ such that the following hold:
	\begin{enumerate}[label=(\alph*)]
	\item\label{BKA:aut} $\rho|_{G_{L}} \otimes \Qbar_p \cong \rho_{\Pi,\iota}$, where $\rho_{\Pi,\iota}$ is the Galois representation attached to $\Pi$ and $\iota$;
	\item\label{BKA:adequate} $\zeta_p \notin L$ and $\rhobar(G_{L(\zeta_p)})$ is adequate.
	\end{enumerate}
Then 
\begin{enumerate}
\item\label{BKA:BK} 
the geometric Bloch--Kato Selmer group 
	\[ H_g^1(F(S)/F^+,\ad(\rho)) := \ker\big( H^1(F(S)/F^+,\ad(\rho)) \rightarrow \prod_{v|p} 
		H^1(F_v^+,B_{\dR} \otimes_{\Q_p} \ad(\rho)) \big) \]
is trivial.
\end{enumerate}
Moreover,
\begin{enumerate}[resume]
\item\label{BKA:H2} $H^2(F(S)/F^+,\ad(\rho)) = 0$;
\item\label{BKA:H1} letting $H_g^1(F_v^+,\ad(\rho)) := \ker(H^1(F_v^+,\ad(\rho)) \rightarrow H^1(F_v^+,B_{\dR} \otimes_{\Q_p} \ad(\rho)))$ for each $v|p$ in $F^+$, the natural map
	\[ H^1(F(S)/F^+,\ad(\rho)) \lra \prod_{v|p} H^1(F_v^+,\ad(\rho))/H_g^1(F_v^+,\ad(\rho)) \]
is an isomorphism.
\end{enumerate}
\end{mainthm}

We refer the reader to the notation section below for any of the notation with which the reader is not familiar, 
to \S\ref{sec:autgalrep} for a discussion of regular algebraic polarizable cuspidal automorphic representations and their associated Galois representations, and to \ref{def:GLadequate} for the definition of an adequate subgroup of $\GL_d(\F)$. 
If $p>2(d+1)$, then any subgroup of $\GL_d(\F)$ acting absolutely irreducibly on $\F^d$ is adequate by a theorem of Guralnick--Herzig--Taylor--Thorne \cite{ThorneAdequate}*{Theorem~A,9}, and the assumption of potential automorphy is satisfied in many cases by work of Barnet-Lamb--Gee--Geraghty--Taylor \cite{BLGGT}*{Theorem~4.5.1}. 
As an example, in \S\ref{sec:sympow} we apply this to certain twists of even symmetric powers of Galois representations associated to elliptic modular forms.
Parts~\ref{BKA:H2} and~\ref{BKA:H1} of Theorem~\ref{thm:BKA} follow from part~\ref{BKA:BK}, using an argument of Kisin (see~\ref{thm:cohom}), and this relies on the ``numerical coincidence" discussed in \cite{CHT}*{\S1}. 
It seems likely that the assumption that $p>2$ can be removed using recent work of Thorne \cite{Thorne2adic}, but we have not checked the details.
Theorem~\ref{thm:BKA} will follow (see \ref{sec:CMThm}) from a slight variant (\ref{thm:BK}). 

A similar result (as well as a result similar to Theorem~\ref{thm:thmC} below) was independently obtained by Breuil--Hellmann--Schraen \cite{BHS}*{Corollaires~4.12 and 4.13}.  
Our results on the Bloch--Kato Selmer group (and on universal deformation rings, c.f. Theorem~\ref{thm:thmC}) are more general, as we make no assumption on the local factors $\pi_v$ of the automorphic representation at places $v|p$, whereas they assume $\pi_v$ is unramified when $v|p$. 
On the other hand, they deduce their results as an application of the construction and study of an interesting object they call a \emph{patched eigenvariety}, and this construction has other applications; for example, to a conjecture of Breuil on locally analytic vectors in completed cohomology \cite{BHS}*{Corollaire~4.4}.

For $\rho$ as in the statement of the theorem, $H_g(F_v^+,\ad(\rho)) = H_f^1(F_v^+,\ad(\rho))$ for any $v|p$ in $F^+$, where 
	\[ H_f^1(F_v^+,\ad(\rho)) : = \ker( H^1(F_v^+,\ad(\rho)) \rightarrow H^1(F_v^+,B_{\cris}\otimes_{\Q_p}\ad(\rho))) \]
is the more common local Bloch--Kato Selmer group (see \ref{rmk:gfsame}), 
so Theorem~\ref{thm:BKA} remains unchanged replacing $H_g^1$ with $H_f^1$ everywhere. 
With this in mind, the reader should compare Theorem~\ref{thm:BKA} with the results of \cite{ChenevierFern}*{\S6}. 
In particular, note that we make no assumption on the restriction of $\rho$ to decomposition groups above $p$ (besides what is naturally implied by assumption \ref{BKA:aut}).

In \S\ref{sec:mainthm}, using cyclic base change we deduce from Theorem~\ref{thm:BKA} a similar theorem for totally real fields. 
To state it, we set up some notation. 
Let $F^+(S)$ be the maximal extension of $F^+$ unramified outside $S$ and all places above $\infty$. 
Let
	\[ \rho : \Gal(F^+(S)/F^+) \lra \GL_d(E) \]
be a continuous, absolutely irreducible representation. 
We assume that $\rho$ satisfies one of the following:
\begin{enumerate}
	\item[(GO)]\label{B:orthogonal} $\rho$ factors through a map $\Gal(F^+(S)/F^+) \rightarrow \GO_d(E)$, that we again denote by $\rho$, with totally even multiplier character;
	\item[(GSp)]\label{B:symplectic} $d$ is even and $\rho$ factors through a map $\Gal(F^+(S)/F^+) \rightarrow \GSp_d(E)$, that we again denote by $\rho$, with totally odd multiplier character.
\end{enumerate}
We will refer to the first as the GO-case, and the second as the GSp-case. 
If we are in the GO-case, then we let $\ad(\rho)$ and $\ad^0(\rho)$ denote the Lie algebra $\mathfrak{go}_d(E)$ of $\GO_d(E)$ and sub-Lie algebra $\mathfrak{so}_d(E)$, respectively, with the adjoint $\Gal(F^+(S)/F^+)$-action $\ad \circ \rho$. 
If we are in the GSp-case, then we let $\ad(\rho)$ and $\ad^0(\rho)$ denote the Lie algebra $\mathfrak{gsp}_d(E)$ of $\GSp_d(E)$ and sub-Lie algebra $\mathfrak{sp}_d(E)$, respectively, with the adjoint $\Gal(F^+(S)/F^+)$-action $\ad \circ \rho$. 

\begin{mainthm}\label{thm:thmB}
Assume $p>2$. 
Assume there is a finite extension $L^+/F^+$ of totally real fields, a regular algebraic polarizable cuspidal automorphic representation $\pi$ of $\GL_d(\A_{L^+})$, and an isomorphism $\iota : \Qbar_p \xrightarrow{\sim}\C$ such that:
	\begin{enumerate}[label=(\alph*)]
	\item $\rho|_{G_{L^+}} \otimes \Qbar_p \cong \rho_{\pi,\iota}$, where $\rho_{\pi,\iota}$ is the Galois representation attached to $\pi$ and $\iota$;
	\item $\rhobar(G_{L^+(\zeta_p)})$ is adequate.
	\end{enumerate}
Then 
\begin{enumerate}
\item\label{thmB:BK} the geometric Bloch--Kato Selmer group 
	\[ H_g^1(F^+(S)/F^+,\ad(\rho)) := \ker\big( H^1(F^+(S)/F^+,\ad(\rho)) \rightarrow \prod_{v|p} 
		H^1(F_v^+,B_{\dR} \otimes_{\Q_p} \ad(\rho)) \big) \]
is trivial. 
\end{enumerate}
Moreover,
\begin{enumerate}[resume]
\item\label{thmB:H2} $H^2(F^+(S)/F^+,\ad^0(\rho)) = 0$;
\item\label{thmB:H1} letting $H_g^1(F_v^+,\ad^0(\rho)) := \ker(H^1(F_v^+,\ad^0(\rho)) \rightarrow H^1(F_v^+,B_{\dR} \otimes_{\Q_p} \ad^0(\rho)))$ for each $v|p$ in $F^+$, the natural map
	\[ H^1(F^+(S)/F^+,\ad^0(\rho)) \lra \prod_{v|p} H^1(F_v^+,\ad^0(\rho))/H_g^1(F_v^+,\ad^0(\rho)) \]
is an isomorphism.
\end{enumerate}
\end{mainthm}

As in \cite{KisinGeoDefs}*{\S8}, we use vanishing of the geometric Bloch--Kato adjoint Selmer group to deduce smoothness of a certain universal deformation ring at automorphic points, something we now explain. 
Let $\calG_d$ denote the group scheme over $\Z$ that is the semidirect product 
	\[ \calG_d^0 \rtimes \{1,\jmath\} = (\GL_d \times \GL_1) \rtimes \{1,\jmath\}, \]
where $\jmath (g,a) \jmath = (a \transp g^{-1},a)$. 
There is a character $\nu : \mathcal{G}_d \rightarrow \GL_1$ defined by by $\nu(g,a) = a$ and $\nu(\jmath) = -1$. 
Fix a continuous homomorphism 
	\[ \rbar : \Gal(F(S)/F^+) \lra \calG_d(\F) \]
inducing an isomorphism $\Gal(F/F^+) \xrightarrow{\sim} \calG_d(\F)/\calG_d^0(\F)$, as well as a totally odd \mbox{de Rham} character $\mu : \Gal(F(S)/F^+) \rightarrow \calO^\times$ such that $\nu \circ \rbar = \mu \bmod {\frakm_{\calO}}$. 
Let $\rhobar : \Gal(F(S)/F) \rightarrow \GL_d(\F)$ be the composite of $\rbar|_{G_F} : \Gal(F(S)/F) \rightarrow \calG_d^0(\F)$ with the projection $\calG_d^0(\F) \rightarrow \GL_d(\F)$, and assume $\rhobar$ is absolutely irreducible. 
If $p>2$, then with this data $\calS = (F/F^+,S,\calO,\rbar,\mu)$, there is a complete Noetherian local commutative $\calO$-algebra $R_{\calS}$ with residue field $\F$, such that $R_{\calS}$ represents the set-valued functor on the category of complete Noetherian local commutative $\calO$-algebras $A$ with residue field $\F$, that sends $A$ to the set deformations of $\rbar$ to $A$ with multiplier $\mu$, i.e. the set of $1+\M_d(\frakm_A)$-conjugacy classes of homomorphisms 
	\[ r : \Gal(F(S)/F^+) \lra \calG_d(A) \]
satisfying $r \bmod {\frakm_A} = \rbar$ and $\nu\circ r = \mu$ (see \S\ref{sec:global}).

\begin{mainthm}\label{thm:thmC}
Assume $p>2$. 
Let $x$ be a closed point of $\Spec R_{\calS}[1/p]$ with residue field $k$, and let
	\[ r_x : \Gal(F(S)/F^+) \lra \calG_d(k) \]
be the pushforward of (some homomorphism representing) the universal deformation of $\rbar$ via $R_{\calS}[1/p] \xrightarrow{x} k$. 
Let $\rho$ denote the composite of $r|_{G_L} : \Gal(F(S)/F) \rightarrow \calG_d^0(k)$ with the projection $\calG_d^0(k) \rightarrow \GL_d(k)$.

Assume there is a finite extension $L/F$ of CM fields, a regular algebraic polarizable cuspidal automorphic representation $\Pi$ of $\GL_d(\A_{L})$, and an isomorphism $\iota : \Qbar_p \xrightarrow{\sim}\C$ such that the following hold:
	\begin{enumerate}[label=(\alph*)]
	\item\label{smooth:aut} $\rho|_{G_L} \otimes \Qbar_p\cong \rho_{\Pi,\iota}$;
	\item\label{smooth:adequate} $\zeta_p \notin L$ and $\rhobar(G_{L(\zeta_p)})$ is adequate.
	\end{enumerate}
Then the localization and completion $(R_{\calS})_x^\wedge$ is formally smooth over $k$ of dimension $\frac{d(d+1)}{2}[F^+:\Q]$.
\end{mainthm}

Theorem~\ref{thm:thmC} is deduced in \ref{sec:smooth} as an immediate consequence of \ref{thm:BK} and \ref{thm:globalsmooth}. 

A related application of Theorem~\ref{thm:BKA} is to the geometry of unitary eigenvarieties. 
Bella\"{i}che and Chenevier showed that at certain classical automorphic points on a unitary eigenvariety (see \cite{BellaicheChenevierBook}*{\S7.6.2}), the weight map is \'{e}tale and  that a ``refined deformation ring" (see \cite{BellaicheChenevierBook}*{Definition~7.6.2 and Proposition 7.6.3}) is isomorphic to the completed local ring of the structure sheaf of the eigenvariety at these points \cite{BellaicheChenevierBook}*{Corollary~7.6.11}, provided that \cite{BellaicheChenevierBook}*{Conjecture~7.6.5~(C1)} holds. 
Theorem~\ref{thm:BKA} implies \cite{BellaicheChenevierBook}*{Conjecture~7.5.6~(C1)} if the representation denoted $V$ there further  satisfies the assumptions of Theorem~\ref{thm:BKA}.   
Thus, Theorem~\ref{thm:BKA} establishes the conjectures of \cite{BellaicheChenevierBook}*{\S7.6} at points satisfying the assumptions of Theorem~\ref{thm:BKA}. 
We refer the reader to \cite{BellaicheChenevierBook}*{\S7.6} for more details and for precise statements. 
These conjectures of Bella\"{i}che--Chenevier were also proved by Breuil--Hellmann--Schraen \cite{BHS}*{Corollaire~4.11}, under essentially the same hypotheses. 
We also note that many cases of this were proved by completely different methods by Chenevier \cite{ChenevierFern}*{Theorems~4.8 and~4.10}.  
%
%

An important step in the proofs of our main theorems is the fact that the points coming from regular algebraic polarizable cuspidal automorphic representations on certain local deformation rings are smooth. 
By local-global compatibility, the Weil--Deligne representations attached to their local factors satisfy the property that they do not admit any nontrivial morphisms to their Tate twist. 
We say below (\ref{def:genericWD}) that such Weil--Deligne representations are \emph{generic}. 
We show (see \ref{thm:smoothlp}) that this property characterizes smooth points on local potentially semistable deformation rings. 
As this may be of independent interest, we state it here. 

\begin{mainthm}\label{thm:thmD}
Let $K$ be a finite extension of $\Q_p$, let
	\[ \rhobar : \Gal(\overline{K}/K) \lra \GL_d(\F) \]
be a continuous homomorphism, and let $R_{\rhobar}^\square$ be the universal framed deformation ring for $\rhobar$. 
Fix a Galois type $\tau$ and a $p$-adic Hodge type $\mathbf{v}$. 
Let $R_{\rhobar}^\square(\tau,\mathbf{v})$ denote Kisin's $\calO$-flat potentially semistable framed deformation ring with fixed Galois type $\tau$ and $p$-adic Hodge type $\mathbf{v}$. 
Let $x$ be a closed point of $\Spec R_{\rhobar}^\square(\tau,\mathbf{v})[1/p]$ with residue field $k$. 
Let $\rho_x : \Gal(\overline{K}/K) \rightarrow \GL_d(k)$ be the pushforward of the universal lift via $R_{\rhobar}^\square[1/p] \rightarrow R_{\rhobar}^\square(\tau,\mathbf{v})[1/p] \xrightarrow{x} k$, and let $\WD(\rho_x)$ be its associated Weil--Deligne representation. 

Then $R_{\rhobar}^\square(\tau,\mathbf{v})[1/p]$ is smooth at $x$ if and only if there is no nonzero map $\WD(\rho_x) \rightarrow \WD(\rho_x)(1)$ of Weil--Deligne representations.
\end{mainthm}

We refer the reader to \S\S\ref{sec:WD} and~\ref{sec:local} for the definitions of any of the terms or object in the statement of Theorem~\ref{thm:thmD} with which the reader is not familiar. 
The idea of the proof of Theorem~\ref{thm:thmD} is quite simple: mimic the $\ell \ne p$ case. 
If $K$ is a finite extension of $\Q_{\ell}$ with $\ell \ne p$, then the smooth closed points in the generic fibre of the universal framed deformation ring 
are precisely the ones whose Weil--Deligne representation satisfy this genericity hypothesis (as observed in \cite{BLGGT}*{Lemma~1.3.2}, see also \ref{thm:localdefl} below). 
Indeed, a result of Gee \cite{GeeTypes}*{Theorem~2.1.6} gives the dimension of the generic fibre of this framed deformation ring, so it suffices to check when the tangent space, which is related to an $H^1$, has strictly larger dimension, and this can be related to our genericity condition using local Euler characteristic and local Tate duality. 
The proof of Theorem~\ref{thm:thmD} is similar, replacing Gee's dimension result with Kisin's \cite{KisinPssDefRing}*{Theorem~3.3.4}, the $H^1$ with the local geometric Bloch--Kato Selmer group $H_g^1$, and local Euler characteristic and local Tate duality with Bloch and Kato's local dimension formulae and local duality results \cite{BlochKato}*{\S3}. 
In fact, for the proofs of Theorems~\ref{thm:BKA}, \ref{thm:thmB}, and~\ref{thm:thmC}, we will only need the direction ``generic implies smooth", but the author has decided to include the converse here because he finds it interesting in its own right: 
the criterion is the same as in the $\ell \ne p$ case, exhibiting a sort of ``independence of $p$" phenomena for the geometry of potentially semistable deformation rings. 

\addtocontents{toc}{\SkipTocEntry}
\subsection*{Outline} The paper is organized as follows.

In \S\ref{sec:deftheory}, we recall and prove the relevant facts regarding the deformation theory of Galois representations. 
We first recall some properties of Weil--Deligne representations in \ref{sec:WD}; namely the construction of Weil--Deligne representations attached to local Galois representations and their connection with smooth admissible representations of $\GL_d$. 
In \S\ref{sec:local}, we treat the local theory of Galois deformations with an emphasis on describing the smooth points in the generic fibre of local deformation rings; in particular, we prove Theorem~\ref{thm:thmD} (\ref{thm:smoothlp}).  
The global theory is then discussed in \S\ref{sec:global}, the main point being to use Kisin's technique of analyzing the generic fibre of deformation rings to connect them to the Bloch--Kato Selmer group (\ref{thm:smallglobalcomplete}), and to prove some dimension and smoothness results that are necessary for the proof of Theorem~\ref{thm:thmC} (\ref{thm:globalsmooth}). 

We discuss the automorphic theory necessary in \S\ref{sec:auttheory}; everything here is standard. 
We first recall the properties of Galois representations associated to regular algebraic polarized cuspidal automorphic representations of CM and totally real fields in \S\ref{sec:autgalrep}. 
We then describe their connection to automorphic forms on definite unitary groups in \S\ref{sec:unitary}, the main point being to show how one can construct a certain module of automorphic forms with an action of a certain Galois deformation ring, and that the vanishing of the Bloch--Kato Selmer group in question is implied by a freeness property (\ref{thm:heckedef}).

We conclude our efforts in \S\ref{sec:mainsec}, proving the main theorem (\ref{thm:BK}) in \S\ref{sec:mainthm}, and deducing Theorems~\ref{thm:BKA}, \ref{thm:thmB}, and~\ref{thm:thmC} from the introduction in \S\ref{sec:introthms}.
Finally, in \S\ref{sec:sympow} we given an application to certain twists of symmetric powers of modular Galois representations.

\subsubsection*{Acknowledgments}

I would like to thank Mark Kisin for helpful conversations and for helpful comments on an earlier version of this paper. 
I would like thank Rebecca Bellovin and Frank Calegari for helpful conversations, and Florian Herzig for helpful comments on an earlier draft of this paper.
I would like to thank John Bergdall for pointing out the application to the work of Bella\"{i}che and Chenevier, Chandrashekhar Khare for pointing out that I could deduce the vanishing of the $H^2$, and David Hansen for pointing out the application to certain twists of symmetric powers of modular Galois representations. 
I would like to thank Matthew Emerton for bringing my attention to the article \cite{KisinGeoDefs}, which started my interest in this project. 
Finally, I would like to thank the referee for their helpful comments and corrections.

\section*{Notation and conventions}

If $F$ is a number field and $v$ is a place of $F$, we denote by $F_v$ the completion of $F$ at $v$, and if $v$ is non-Archimedean we denote the ring of integers by $\calO_{F_v}$. 
We let $\A_F$ denote the ring of adeles of $F$, and  $\A_F^\infty$ the ring of finite adeles. 

If $K$ is any field with a fixed algebraic closure $\overline{K}$, we denote by $G_K$ the absolute Galois group $\Gal(\overline{K}/K)$. 
If $K$ is a non-Archimedean local field, we denote by $I_K$ the inertia subgroup, by $W_K$ the Weil group, and by $\Frob_K$ the geometric Frobenius in $G_K/I_K \cong W_K/I_K$. In the case that $K$ is the completion of a number field $F$ at a finite place $v$, we write $G_v$, $I_v$, and $\Frob_v$ for $G_{F_v}$, $I_{F_v}$, and $\Frob_{F_v}$, respectively. 
If $L/F$ is a Galois extension of number fields inside some fixed algebraic closure $\overline{F}$ of $F$, and $S$ is a finite set of finite places of $F$, we let $L(S)$ denote the maximal extension of $L$ that is unramified outside of any of the places in $L$ above those in $S$ and the Archimedean places. 
A CM extension of a totally real field is always assumed to be imaginary. 
Given a CM field $F$ with maximal totally real subfield $F^+$, we denote by $\delta_{F/F^+}$ the nontrivial $\{\pm 1\}$-valued character of $\Gal(F/F^+)$. 
Given a finite separable extension of fields $L/F$, we write $\Nm_{L/F}$ for the norm from $L$ to $F$.

If $K$ is a non-Archimedean local field, we let $\Art_K : K^\times \xrightarrow{\sim} W_K$ be the Artin reciprocity map normalized so that uniformizers are sent to geometric Frobenius elements. 
For a number field $F$, we let $\Art_F : F^\times \backslash \A_F^\times \rightarrow G_F^{\ab}$ be $\Art_F = \prod_{v} \Art_{F_v}$.
For any $d\ge 1$, we let $\rec_K$ be the Local Langlands reciprocity map that takes irreducible admissible representations of $\GL_d(K)$ to Frobenius semi-simple Weil--Deligne representations, normalized as in \cite{HarrisTaylor} and \cite{HenniartLL}. 
We then let $\recT_K$ be given by 
$\recT_K(\pi) = \rec_K(\pi \otimes \abs{\cdot}^{\frac{1-d}{2}})$.

If we are given an isomorphism $\iota : K \xrightarrow{\sim} L$ of fields, and $F\hookrightarrow K$ is an embedding of fields, we write $\iota\lambda$ for $\iota\circ\lambda$. If $\lambda = (\lambda_i)_{i\in I}$ is a tuple of field embeddings $F_i \hookrightarrow K$, then we write $\iota\lambda$ for the tuple $(\iota\lambda_i)_{i\in I}$. If  $r : G \rightarrow \Aut_K(V)$ is a representation of a group $G$ on a $K$-vector space $V$, then we will denote by $\iota r$ the representation of $G$ on the $L$-vector space $V\otimes_{K,\iota}L$. 

If $G$ is a group, $A$ is a commutative ring, and $\rho : G \rightarrow \GL_n(A)$ is a homomorphism, then we will let $V_\rho$ denote the representation space of $\rho$, i.e. $V_\rho = A^n$ with the $A[G]$-module structure coming from $\rho$.

We denote by $\epsilon$ the $p$-adic cyclotomic character. 
We let $B_{\mathrm{cr}}$, $B_{\st}$, and $B_{\dR}$ denote Fontaine's rings of crystalline, semistable, and \mbox{de Rham} periods, respectively. 
We will frequently use the Berger--Andr\'{e}--Kedlaya--Mebkhout Theorem \cite{BergerReppadicDifEq}*{Th\'{e}or\`{e}me~0.7}, that \mbox{de Rham} representation are potentially semistable, without comment. 
We use covariant $p$-adic Hodge theory, and normalize our Hodge--Tate weights so that the Hodge--Tate weight of $\epsilon$ is $-1$. 
If $K$ and $E$ are two algebraic extensions of $\Q_p$ with $K/\Q_p$ finite, $\tau : K \hookrightarrow E$ is a continuous embedding, and $\rho : G_K \rightarrow \GL(V) \cong \GL_d(E)$ is a \mbox{de Rham} representation, we will write $\HT_\tau(\rho)$ for the multiset of $d$ Hodge--Tate weights with respect to $\tau$. 
Specifically, an integer $i$ appears in $\HT_\tau(\rho)$ with multiplicity equal to the $E$-dimension of the $i$th graded piece of the $d$-dimensional filtered $E$-vector space $D_{\dR}(\rho)\otimes_{(K\otimes_{\Q_p}E)} E$, where $D_{\dR}(\rho) = (B_{\dR}\otimes_{\Q_p} V_{\rho})^{G_K}$ and we view $E$ as a $K\otimes_{\Q_p} E$-algebra via $\tau \otimes 1$.  
We will say a continuous representation $\rho : G_F \rightarrow \GL(V) \cong \GL_d(E)$ of the Galois group of a number field $F$ is \mbox{de Rham}, resp. semistable, resp. crystalline, if $\rho|_{G_v}$ is so for every $v|p$ in $F$. 
%


If $A$ is a commutative local ring, we will denote by $\frakm_A$ its maximal ideal. If $A$ is a commutative ring and $x: A \rightarrow D$ is a homomorphism with $D$ a domain, then we denote by $A_x$ the localization of $A$ at $\ker(x)$, and $A_x^\wedge$ the localization and completion of $A$ at $\ker(x)$. If $A$ is a commutative ring and $x \in \Spec A$ has residue field $k_x$, we again denote by $x$ the map $x : A \rightarrow k_x$.

If $k$ is a characteristic $0$ field and $R$ is a commutative $k$-algebra, we will say $x\in \Spec R$ is formally smooth if $k \rightarrow R_x$ is formally smooth. 
This is equivalent to $k \rightarrow R_x^\wedge$ being formally smooth, which is equivalent to $R_x^\wedge$ being isomorphic to a power series ring over it's residue field, since $k$ has characteristic $0$. 
We will use these equivalences without further comment.

If $B$ is a local commutative Noetherian ring, we let $\CNL_B$ be the category of complete local commutative Noetherian $B$-algebras $A$ such that the structure map $B \rightarrow A$ induces an isomorphism $B/\frakm_B \xrightarrow{\sim} A/\frakm_A$, and whose morphisms are local $B$-algebra morphisms. 
We will refer to an object, resp. a morphism, in $\CNL_B$ as a $\CNL_B$-algebra, resp. a $\CNL_B$-morphism. 
The full subcategory of Artinian objects is denoted $\Ar_B$. 

If $G$ is a topological group, and $W$ is a topological $G$-module, the cohomology groups $H^i(G,W)$ are always assumed to be the continuous cohomology groups, i.e. the cohomology groups computed with continuous cochains. 
If $M/L$ is a Galois extension, and $W$ is a topological $\Gal(M/L)$-module, we write $H^i(M/L,W)$ for $H^i(\Gal(M/L),W)$. 
If $\overline{L}$ is an algebraic closure of $L$, we write $H^i(L,W)$ for $H^i(\overline{L}/L, W)$. 
If $K$ is a finite extension of $\Q_p$, and $W$ is a finite dimensional $\Q_p$-vector space with a continuous $\Q_p$-linear $G_K$-action, we set
	\[ H_g^1(K,W) := \ker( H^1(K,W) \rightarrow H^1(K,B_{\dR}\otimes_{\Q_p} W)). \]
If $F$ is a number field, $M$ is Galois extension of $F$, and $W$ is a finite dimensional $\Q_p$-vector space with a continuous $\Q_p$-linear $\Gal(M/F)$-action, we set
	\[ H_g^1(M/F,W) := \ker \big( H^1(M/F,W) \rightarrow \prod_{v|p} H^1(F_v,B_{\dR}\otimes_{\Q_p} W) \big). \]


\section{Deformation theory}\label{sec:deftheory}

Throughout this section $E$ will denote a finite extension of $\Q_p$ with ring of integers $\calO$ and residue field $\F$. 

\subsection{Weil--Deligne representations}\label{sec:WD}
Let $\ell$ be a rational prime, and let $K$ be a finite extension of $\Q_\ell$ with ring of integers $\calO_K$ and uniformizer $\varpi_K$. 
Let $q$ denote the cardinality of the residue field of $K$, Let $\abs{\cdot}$ denote the absolute value on $K$ normalized so that $\abs{\varpi_K} = q^{-1}$. 
For $w \in W_K$, we will write $\abs{w}$ for $\abs{\Art_K^{-1}(w)}$. 
If $k$ is a field, and $\rho : G_K \rightarrow \GL_d(k)$ is a homomorphism, we let $\ad(\rho)$ denote $\mathfrak{gl}_d(k)$ with the adjoint action $\ad\circ \rho$ of $G_K$. 
Note that $\ad(\rho) \cong \Hom_k(V_\rho,V_\rho)$ and $\ad(\rho)(1) \cong \Hom_k(V_\rho,V_\rho(1))$ as $k[G_K]$-modules. 

\subsubsection{}\label{sec:genericWD} 
We recall some basics of Weil--Deligne representations (see \cite{TateCorvallis}*{\S4}). Given a characteristic $0$ field $\Omega$, a \emph{Weil--Deligne representation over} $\Omega$ is a pair $(r,N)$, where $r : W_K \rightarrow \GL(V) \cong \GL_d(\Omega)$ is a representation of $W_K$ on a finite dimensional $\Omega$-vector space $V$ with open kernel, and $N\in \End_\Omega(V)$ is nilpotent, such that $r(w)Nr(w)^{-1} = \abs{w}N$ for all $w\in W_K$. A morphism of Weil--Deligne representations $(r_1,N_1) \rightarrow (r_2,N_2)$ is a an $\Omega$-linear morphism that intertwines the $r_i$ and the $N_i$. A Weil--Deligne representation $(r,N)$ is called \emph{Frobenius-semisimple} if $r$ is semisimple (equivalently, if $r(\Phi)$ is semisimple for $\Phi \in W_K$ a lift of the Frobenius). Given a Weil--Deligne representation $(r,N)$, we will denote by $(r,N)^{\Fss}$ its Frobenius-semisimplification, i.e. $(r,N)^{\Fss} = (r^{\mathrm{ss}},N)$. 
Given a Weil--Deligne representation $(r,N)$, we let $(r,N)(1) = (r(1),N)$ be the Weil--Deligne representation with $r(1)(w) = \abs{w}r(w)$. If $\iota\in\Aut(\Omega)$, we let $\iota(r,N) = (\iota r, \iota N)$ denote the Weil Deligne representation obtained by change of scalars via $\iota : \Omega \xrightarrow{\sim}\Omega$ (this is again a Weil--Deligne representation since $\abs{w}\in \Q$ for all $w\in W_K$).

\begin{defn}\label{def:genericWD} We say a Weil--Deligne representation $(r,N)$ is \emph{generic} if there is no nontrivial morphism $(r,N) \rightarrow (r(1),N)$.
\end{defn}

If $\pi$ is an irreducible admissible representation of $\GL_d(K)$ over $\C$ and $\iota\in \Aut(\C)$, then $\recT_K(\iota\pi) = \iota\recT_K(\pi)$ (this is explained when $d=2$ in \cite{BushnellHenniartLLGL2}*{\S35} and the argument there generalizes using \cite{BushnellHenniartHD2}*{Theorem~3.2} and the converse theorems of \cite{HenniartEpsilon}). If $\Omega$ is field isomorphic to $\C$ (as abstract fields), we get a bijection, again denoted by $\recT_K$, between isomorphism classes of irreducible admissible representations of $\GL_d(K)$ over $\Omega$ and isomorphism classes of $d$-dimensional Frobenius semi-simple Weil--Deligne representations over $\Omega$ by fixing any isomorphism $\iota : \Omega \xrightarrow{\sim} \C $ and setting $\recT_K(\pi) = \iota^{-1}\recT_K(\iota\pi)$,
and this is independent of the choice of $\iota$.

\begin{lem}\label{thm:automorphicWD}
Let $\pi$ be an irreducible smooth admissible representation of $\GL_d(K)$ on an $\Omega$-vector space, with $\Omega$ a field (abstractly) isomorphic to $\C$.
Then $\recT_K(\pi)$ is generic if and only if $\pi$ is generic.
\end{lem}

\begin{proof}
This is essentially identical to \cite{BLGGT}*{Lemma~1.3.2(1)}. We give the details. It suffices to consider the case $\Omega = \C$. Since $\pi$ is generic if and only if $\pi\otimes\abs{\cdot}^{\frac{1-d}{2}}$ is generic, it is equivalent to show that $\pi$ is generic if and only if $\rec_K(\pi)$ is generic. Note that if $(r,N) = \rec_K(\pi)$, then $(r(1),N) = \rec_K(\pi\otimes\abs{\cdot})$.

We will use the notation and terminology of~\cite{HarrisTaylor}*{\S1.3}. There are positive integers $s_i,d_i$ for $i = 1,\ldots, t$ with $d = d_1s_1 + \cdots + d_ts_t$ and irreducible supercuspidal representations $\pi_i$ of $\GL_{d_i}(K)$ such that
	\[ \pi \cong \Sp_{s_1}(\pi_1) \boxplus \cdots \boxplus \Sp_{s_t}(\pi_t), \]
and the multiset $\{(s_1,\pi_1),\ldots,(s_t,\pi_t)\}$ is uniquely determined by $\pi$. 
By abuse of notation, we also denote by $\Sp_s$ the $s$-dimensional Weil--Deligne representation $(r,N)$ on a complex vector space with basis $e_0,\ldots,e_{s-1}$, where $r(w) = \abs{w}^i e_i$ for each $i = 0,\ldots,s-1$, and $Ne_i = e_{i+1}$ for each $i = 0,\ldots, s-2$ and $N e_{s-1}=0$. 
Then (see \cite{HarrisTaylor}*{Theorem~VII.2.20} and the discussion preceding it)
	\[ \rec_K (\pi) = (\rec_K(\pi_1)\otimes \Sp_{s_1})\oplus \cdots \oplus (\rec_K(\pi_t)\otimes\Sp_{s_t}),
	\]
and $\rec_K(\pi)$ is nongeneric if and only if 
	\begin{equation}\label{eqn:genericnonzero} \Hom_{\WD}(\rec_K(\pi_i)\otimes \Sp_{s_i}, \rec_K(\pi_j\otimes\abs{\cdot}) \otimes \Sp_{s_j}) \ne 0
	\end{equation}
for some $i,j$. Since $\rec_K(\pi_i)$ is absolutely irreducible (as $\pi_i$ is supercuspidal), it is easy to check that
\eqref{eqn:genericnonzero} holds if and only if 
$\pi_i \cong \pi_j \otimes\abs{\cdot}^a$ with $s_j - s_i < a \le s_j$. In the notation and terminology of \cite{Zelevinsky} (see \cite{Zelevinsky}*{\S3.1 and \S4.1}), this happens if and only if the segments $[\pi_i,\ldots,\pi_i\otimes\abs{\cdot}^{s_i-1}]$ and $[\pi_j,\ldots,\pi_j\otimes\abs{\cdot}^{s_j-1}]$ are linked, which happens if and only if $\pi$ is non-generic by \cite{Zelevinsky}*{Theorem~9.7} (note \emph{generic} is called \emph{non-degenerate} in \cite{Zelevinsky}).
\end{proof}

\subsubsection{}\label{sec:lnotpWD} Assume that $\ell \ne p$, and let
	\[ \rho : G_K \lra \GL_d(E) \]
be a continuous representation. Following \cite{TateCorvallis}*{\S4.2}, we can attach a Weil--Deligne representation to our fixed $\rho$, that we will denote $\WD(\rho)$, as follows. Fix $\Phi \in G_K$ mapping to the geometric Frobenius in $G_K/I_K$. Fix a surjection $t_p: I_K \rightarrow \Z_p$, and let $\tau_p\in I_K$ be such that $t_p(\tau_p) = 1$. The homomorphism $t_p$ necessarily factors through tame inertia. Write $\rho(\tau_p) = \rho(\tau_p)^{ss}\rho(\tau_p)^u$ with $\rho(\tau_p)^{ss}$ semisimple and $\rho(\tau_p)^u$ unipotent. Set $N = \log(\rho(\tau_p)^{u})$. Then the map $r: W_K \rightarrow \GL_d(E)$ given by
	\begin{equation}\label{eqn:WDrep} r(\Phi^n\sigma) = \rho(\Phi^n\sigma)e^{-t_p(\sigma)N}
	\end{equation}
for $n\in \Z$ and $\sigma\in I_K$, is well-defined with open kernel, and $(r,N)$ is a Weil--Deligne representation. The isomorphism class does not depend on the choices made, and we denote any element in this isomorphism class by $\WD(\rho)$. Moreover, this assignment (which depends on $\Phi$ and $t_p$) gives an equivalence of categories from the category of continuous representations $\rho: G_K \rightarrow \GL_d(E)$ to the full subcategory of Weil--Deligne representations $(r,N)$ on $E^d$ such that $r$ has bounded image. From this we deduce the following lemma.

\begin{lem}\label{thm:lnotpgeneric} Let $\rho : G_K \rightarrow \GL_d(E)$ be a continuous representation. The Weil--Deligne representation $\WD(\rho)$ is generic if and only if $\Hom_{E[G_K]}(V_\rho,V_\rho(1)) = 0$.
\end{lem}

\subsubsection{}\label{sec:pWD} Assume $\ell = p$, and let
	\[ \rho : G_K \lra \GL_d(E) \]
be a continuous potentially semistable representation. Following Fontaine, \cite{Fontaineladic}*{\S1.3 and \S2.3}, we can also associate a Weil--Deligne representation to $\rho$, again denoted $\WD(\rho)$, as follows.

Let $L/K$ be a finite extension, let $G_{L/K} = \Gal(L/K)$, and let $L_0$ be the maximal subfield of $L$ unramified over $\Q_p$. We assume that $E$ contains all embeddings of $L_0$ into an algebraic closure of $E$. A $(\varphi,N,G_{L/K})$-\emph{module} $D$ over $E$ is a finite free $L_0 \otimes_{\Q_p} E$-module together with operators $\varphi$ and $N$, and an action of $\Gal(L/K)$, satisfying the following:
	\begin{itemize}
	\item $N$ is $L_0 \otimes_{\Q_p} E$-linear;
	\item $\varphi$ is $E$-linear and satisfies $\varphi(ax) = \sigma(a)\varphi(x)$ for any $x\in D$ and $a\in L_0$, where $\sigma\in\Gal(L_0/\Q_p)$ is the absolute arithmetic Frobenius;
	\item $N\varphi = p\varphi N$;
	\item the $\Gal(L/K)$-action is $E$-linear and $L_0$-semilinear, and commutes with $\varphi$ and $N$.
	\end{itemize}
Extend the action of $\Gal(L/K)$ to $W_K$ by letting $I_L$ act trivially. For $w\in W_K$, we let $v(w)\in \Z$ be such that the image of $w$ in $W_K/I_K$ is $\sigma^{-v(w)}$. We then define an $L_0\otimes_{\Q_p}E$-linear action, that we denote $r_D$, of $W_K$ on $D$ by $r_D(w) = w\varphi^{v(w)}$. Writing $L_0 \otimes_{\Q_p} E = \prod_{\tau : L_0 \hookrightarrow E} E$, we get a decomposition $D = \prod_{\tau : L_0 \hookrightarrow E} D_\tau$ and an induced $d$-dimensional Weil--Deligne representation $(r_\tau,N_\tau)$ over $E$ on each factor $D_\tau$. The isomorphism class of $(r_\tau, N_\tau)$ is independent of $\tau : L_0 \hookrightarrow E$ (see \cite{BreuilMezardMultMod}*{\S2.2.1}), and we denote any element in its isomorphism class by $\WD(D)$. Moreover, by \cite{BSpLL}*{Proposition~4.1}, this assignment induces an equivalence of categories from $(\varphi,N,G_{L/K})$-modules over $E$ to Weil--Deligne representations over $E$ on which $I_L$ acts trivially. Given a $(\varphi,N,G_{L/K})$-module $D$ over $E$, we let $D(1)$ be the $(\varphi, N, G_{L/K})$-module with the same underlying $L_0\otimes_{\Q_p} E$-module, operator $N$, and $G_{L/K}$-action, but with $\varphi_{D(1)} = p^{-1} \varphi_D$. Note that $\WD(D(1)) = \WD(D)(1)$.

Now choose $L/K$ such that $\rho|_{G_{L}}$ is semistable and such that $E$ contains all embeddings of $L$ into an algebraic closure of $E$ (enlarging $E$ if necessary). We get a $(\varphi, N, G_{L/K})$-module
	\[ D_{\st,L}(\rho) := (B_{\st} \otimes_{\Q_p} V_\rho)^{G_{L}}.
	\]
The isomorphism class of $\WD(D_{\st,L}(\rho))$ does not depend on the choice of $L$ (see \cite{BreuilMezardMultMod}*{\S2.2.1}), and we set $\WD(\rho) = \WD(D_{\st,L}(\rho))$.

\begin{lem}\label{thm:lpgeneric}
Let $\rho : G_K \rightarrow \GL_d(E)$ be a potentially semistable representation. Let $L/K$ be a finite extension such that $\rho|_{G_L}$ is semistable.
\begin{enumerate}
\item\label{lpgeneric:phiN} The Weil--Deligne representation $\WD(\rho)$ is generic if and only if 
there are no nonzero morphisms $D_{\st,L}(\rho) \rightarrow D_{\st,L}(\rho)(1)$ of $(\varphi, N, G_{L/K})$-modules over $E$.
\item\label{lpgeneric:crys} Let $D_{\cris}(\ad(\rho)(1)) = (B_{\cris} \otimes_{\Q_p} \ad(\rho)(1))^{G_K}$ with its induced crystalline Frobenius $\varphi$. Then $\WD(\rho)$ is generic if and only if 
$D_{\cris}(\ad(\rho)(1))^{\varphi = 1} = 0$.
\end{enumerate}
\end{lem}

\begin{proof} Part \ref{lpgeneric:phiN} follows from \cite{BSpLL}*{Proposition~4.1}. We use this to derive part~\ref{lpgeneric:crys}. 
By \cite{FontaineSemiStable}*{\S5.6}, the $(\varphi,N,G_{L/K})$-module $D_{\st,L}(\Hom_{\Q_p}(V_{\rho},V_{\rho}(1)))$ over $\Q_p$ is identified with the $(\varphi,N,G_{L/K})$-module over $\Q_p$ consisting of the $L_0$-vector space of morphisms $D_{\st,L}(\rho) \rightarrow D_{\st,L}(\rho(1))$ with $(\varphi, N, G_{L/K})$-module structure by
	\begin{itemize}
	\item $\varphi f = \varphi\circ f \circ\varphi^{-1}$,
	\item $Nf = N\circ f - f \circ N$,
	\item $\gamma f = \gamma \circ f \circ \gamma^{-1}$, for $\gamma\in G_{L/K}$. 
	\end{itemize}
This identification takes the subspace of elements that commute with $E$ to the subspace of elements that commute with $E$, and we have an isomorphism $D_{\st,L}(\ad(\rho)(1))$ with the space of $L_0\otimes_{\Q_p} E$-morphisms $D_{\st,L}(\rho) \rightarrow D_{\st,L}(\rho(1))$ with the $(\varphi, N, G_{L/K})$-module structure as above. This together with part \ref{lpgeneric:phiN} implies that $\WD(\rho)$ is generic if and only if
	\[
	\{ f \in D_{\st, L}(\ad(\rho)(1))^{G_{L/K}} \mid Nf = 0 \text{ and }\varphi f = f \} = 0.
	\]
The left hand side of this expression is exactly the subspace of $D_{\cris}(\ad(\rho)(1))$ on which $\varphi = 1$. 
\end{proof}

We note that one may have $\Hom_{E[G_K]}(V_\rho,V_\rho(1)) = 0$, but $\WD(\rho)$ nongeneric, for example if $\rho$ is a nonsplit crystalline extension of the trivial character by the cyclotomic character.

\subsection{Local Galois deformation rings}\label{sec:local}
We keep the notation and terminology of the previous subsection.
Fix a continuous representation
	\[ \rhobar : G_K \lra \GL_d(\F). \]
A \emph{lift} of $\rhobar$ to a $\CNL_{\calO}$-algebra $A$ is a continuous homomorphism
	\[ \rho : G_K \lra \GL_d(A) \]
such that $\rho \bmod {\frakm_A} = \rhobar$. 
The set valued functor that sends a $\CNL_{\calO}$-algebra to its set of lifts is representable (see \cite{BockleDefTheory}*{Proposition~1.3}). 
We call the representing object the \emph{universal lifting ring} for $\rhobar$ and denote it by $R_{\rhobar}^\square$. 
We let $\rho^\square : G_K \rightarrow \GL_d(R_{\rhobar}^\square)$ denote the universal lift.

In what follows, if $R$ is a quotient of $R_{\rhobar}^\square$, and $x \in \Spec R[1/p]$ has residue field $k$, we 
let $\rho_x : G_K \rightarrow \GL_d(k)$ denote the specialization of 
$\rho^\square$ via $R^\square[1/p] \rightarrow R[1/p] \xrightarrow{x} k$. 
We then define a \emph{lift} of $\rho_x$ to a $\CNL_k$-algebra $A$ to be a homomorphism 
	\[ \rho : G_K \lra \GL_d(A) \]
such that $\rho \bmod {\frakm_A} = \rho_x$ and such that the induced map $G_K \rightarrow \GL_d(A/\frakm_A^n)$ is continuous for all $n \ge 1$, where we give $A/\frakm_A^n$ the topology as a finite dimensional $k$-vector space.

The proof of our main theorems will rely crucially on Kisin's method for analyzing the generic fibre of universal deformation rings, the linchpin of which is the following result. 

\begin{thm}\label{thm:defcomplete}
Let $x$ be a closed point of $R_{\rhobar}^\square[1/p]$ with residue field $k$. 
\begin{enumerate}
\item\label{defcomplete:rep} The set valued functor that sends a $\CNL_k$-algebra to the set of lifts of $\rho_x$ is represented by the the localization and completion $(R_{\rhobar}^\square)_x^\wedge$ of $R_{\rhobar}$ at $x$.
\item\label{defcomplete:tangent} The tangent space of $\Spec R_{\rhobar}^\square[1/p]$ at $x$ is canonically isomorphic to the space of $1$-cocyles $Z^1(K,\ad(\rho_x))$ of $G_K$ with coefficients in $\ad(\rho_x)$.
\end{enumerate}
%
\end{thm}

\begin{proof} 
Part \ref{defcomplete:rep} is \cite{KisinFinFlat}*{Lemma~2.3.3 and Proposition~2.3.5}. 
In fact, \cite{KisinFinFlat}*{Proposition~2.3.5} goes further by identifying certain groupoids, which 
implies what we want (see \cite{KisinFinFlat}*{\S A.5}). 

Using part~\ref{defcomplete:rep}, it is straightforward to check that the map $Z^1(K,\ad(\rho_x)) \rightarrow \Hom_{\CNL_k}((R_{\rhobar}^\square)_x^\wedge,k[\varepsilon])$ given by $\kappa \mapsto (1+\varepsilon \kappa)\rho_x$ is an isomorphism of $k$-vector spaces.
\end{proof}

\begin{prop}\label{thm:localdefl} Assume $\ell \ne p$.
\begin{enumerate}
\item\label{localdefl:dim} 
$\Spec R_{\rhobar}^\square[1/p]$ is equidimensional of dimension $d^2$.  
\item\label{localdefl:smooth} A closed point $x$ of $\Spec R_{\rhobar}^\square [1/p]$ is smooth if and only if $\WD(\rho_x)$ is generic.
\end{enumerate}
\end{prop}

\begin{proof} 
The fact that $\Spec R_{\rhobar}^\square[1/p]$ is equidimensional of dimension $d^2$ is a result of Gee \cite{GeeTypes}*{Theorem~2.1.6} (see also the discussion preceding Proposition~2.1.4 of \cite{GeeTypes}). 
Let $k$ denote the residue field of $x$. 
Then \ref{thm:defcomplete} implies that $x$ is a smooth point if and only if $\dim_k Z^1(K,\ad(\rho_x)) = d^2$. 
By the local Euler characteristic formula,
	\begin{align*} \dim_k Z^1(K,\ad(\rho_x)) & = \dim_k H^1(K,\ad(\rho_x)) + d^2 - \dim_k H^0(K,\ad(\rho_x))\\
	& = d^2 + \dim_k H^2(K,\ad(\rho_x)),
	\end{align*}
so $(R_{\rhobar}^\square)_x^\wedge$ is formally smooth over $k$ if and only if $H^2(K,\ad(\rho_x)) = 0$. 
The trace pairing on $\ad(\rho_x)$ is perfect, so Tate local duality implies that $H^2(K,\ad(\rho_x)) = 0$ if and only if $H^0(K,\ad(\rho_x)(1)) = 0$. 
This is equivalent to $\Hom_{k[G_K]}(V_{\rho_x},V_{\rho_x}(1)) = 0$, which is equivalent to $\WD(\rho_x)$ being generic by \ref{thm:lnotpgeneric}.
\end{proof}

\subsubsection{}\label{sec:types}
Assume that $\ell = p$. A $d$-dimensional \emph{Galois type} over $E$ is a representation $\tau : I_K \rightarrow \GL(V)\cong \GL_d(E)$ of $I_K$ on a $d$-dimensional $E$-vector space $V$ with open kernel that extends to a representation of $W_K$. An $d$-dimensional $p$-\emph{adic Hodge type} over $E$ is a pair $\mathbf{v} = (D, \{\Fil^i\}_{i\in \Z})$, where $D$ is a free $K \otimes_{\Q_p} E$-module of rank $d$, and $\{\Fil^i\}_{i\in \Z}$ is a decreasing, separated, exhaustive filtration on $D$ by $K \otimes_{\Q_p} E$-submodules. We set $\ad(D) = \End_{(K\otimes_{\Q_p} E)}(D)$ and $\ad(D)^+ = \{ f\in \ad(D) \mid f(\Fil^i) \subseteq \Fil^i \text{ for all } i \in \Z \}$. We will say that a $p$-adic Hodge type $\mathbf{v} = (D,\{\Fil^i\}_{i\in\Z})$ over $E$ is \emph{regular} if, writing $K\otimes_{\Q_p} E \cong \prod_i K_i$ as a product of fields, the $d$-dimensional filtered $K_i$-vector space $D\otimes_{K\otimes E} K_i$ has graded pieces of dimension at most $1$. 
It is straightforward to check that if $\mathbf{v}$ is regular, then $\dim_E \ad(D)/\ad (D)^+ = \frac{d(d-1)}{2}[K:\Q_p]$, and this is maximal.
%

Let $\tau : I_K \rightarrow \GL(V)$ and $\mathbf{v} = (D,\{\Fil^i\}_{i \in \Z})$ be a $d$-dimensional Galois type and $p$-adic Hodge type, respectively, over $E$. 
Let $A$ be a finite $E$-algebra and let $V_A$ be a free $A$-module of rank $d$ with a continuous $A$-linear $G_K$-action such that $V_A$ is a potentially semistable representation. Let $(r_A,N_A)$ be the Weil--Deligne representation attached to $V_A$ (viewed as a representation of $G_K$ on a $d(\dim_E A)$-dimensional $E$-vector space). 
We say that $V_A$ has \emph{Galois type} $\tau$ if $r_A|_{I_K} \cong \tau \otimes_E A$. 
Let $D_{\dR}(V_A) = (B_{\dR} \otimes_{\Q_p} V_A)^{G_K}$ together with its natural filtration induced from the filtration on $B_{\dR}$. 
We say that $V_A$ has $p$-\emph{adic Hodge type} $\mathbf{v}$ if for each $i\in \Z$, there is an isomorphism of $K \otimes_{\Q_p} A$-modules
	\[ \gr^i D_{\dR}(V_A) \cong \gr^i(D) \otimes_E A. \]
We can now state the following fundamental result of Kisin, \cite{KisinPssDefRing}*{Theorem~3.3.4}.

\begin{thm}\label{thm:pdefring}
Fix a $d$-dimensional Galois type $\tau$, and a $d$-dimensional $p$-adic Hodge type $\mathbf{v} = (D,\{\Fil^i\}_{i\in\Z})$ over $E$.
There is an $\calO$-flat quotient $R_{\rhobar}^{\square}(\tau,\mathbf{v})$ of $R_{\rhobar}^{\square}$ 
such that
if $A$ is any finite $E$-algebra, an $E$-algebra morphism $x: R_{\rhobar}^{\square}[1/p] \rightarrow A$ factors through $R_{\rhobar}^{\square}(\tau,\mathbf{v})[1/p]$ if and only if $\rho_x$ is potentially semistable with Galois type $\tau$ and $p$-adic Hodge type $\mathbf{v}$.
	 
Moreover, 
if nonzero, then $\Spec R_{\rhobar}^\square(\tau,\mathbf{v})[1/p]$ is equidimensional of dimension $d^2 + \dim_E \ad(D)/\ad(D)^+$, and admits a open dense formally smooth subscheme.
\end{thm}

We now wish to show the analogue of  part \ref{localdefl:smooth} of \ref{thm:localdefl} for the rings $R_{\rhobar}^\square(\tau,\mathbf{v})$. 
For global applications, we will actually only need the fact that $\WD(\rho_x)$ generic implies $R_{\rhobar}^\square(\tau,\mathbf{v})_x^\wedge$ is formally smooth, but for completeness we include the converse.
Our proof will rely on the following standard lemma,
which we will also need for other purposes later. 

\begin{lem}\label{thm:localBK}
Let $x$ be a closed point of $\Spec R_{\rhobar}^\square (\tau,\mathbf{v})[1/p]$. 
The tangent space of $R_{\rhobar}^\square (\tau,\mathbf{v})[1/p]$
at $x$ is canonically 
isomorphic to 
	\[ Z_g^1(K,\ad(\rho_x))
	 := \ker \big(Z^1(K,\ad(\rho_x)) \rightarrow H^1(K, B_{\dR} \otimes_{\Q_p} \ad(\rho_x)) \big). \]
\end{lem}

\begin{proof} 
Let $k$ denote the residue field of $x$. 
Using \ref{thm:defcomplete},
the tangent space of $(R_{\rhobar}^\square)_x^\wedge$ is canonically isomorphic to $Z^1(K,\ad(\rho_x))$. 
Take $\kappa \in Z^1(K,\ad(\rho_x))$, and let $\rho_\kappa = (1+\varepsilon \kappa)\rho_x : G_K \rightarrow \GL_d(k[\varepsilon])$ be the corresponding lift. 
The cocycle $\kappa$ dies in $H^1(K, B_{\dR} \otimes_{\Q_p} \ad(\rho_x))$ if and only if there is a $G_K$-equivariant isomorphism 
	\[ B_{\dR} \otimes_{\Q_p} V_{\rho_\kappa} \cong B_{\dR} \otimes_{\Q_p} (V_{\rho_x}\otimes_k k[\varepsilon]) \cong
	(B_{\dR} \otimes_{\Q_p} V_{\rho_x}) \otimes_k k[\varepsilon], \]
and this happens if and only if $\rho_\kappa$ is potentially semistable with $p$-adic Hodge type $\mathbf{v}$. 
Choosing an extension $L/K$ for which $\rho_\kappa$ is semistable and using the exactness of $D_{\st,L}$ (see \cite{FontaineSemiStable}*{Th\'{e}or\`{e}m~5.1}), we see that the Galois type of $\rho_\kappa$ is an extension of $\tau$ by itself. 
Since $\tau$ is a representation of a finite group in characteristic $0$, it necessarily splits and $\rho_\kappa$ has Galois type $\tau$. 
Hence, $\kappa$ lies in the kernel of $Z^1(K,\ad(\rho_x)) \rightarrow  H^1(K, B_{\dR} \otimes_{\Q_p} \ad(\rho_x))$ if and only if the lift $\rho_\kappa$ is potentially semistable of Galois type $\tau$ and $p$-adic Hodge type $\mathbf{v}$. 
By \ref{thm:pdefring}, the subspace of such elements is the tangent space of $R_{\rhobar}^\square (\tau,\mathbf{v})_x^\wedge$.
\end{proof}

\begin{lem}\label{thm:adlem}
Let $k$ be a finite extension of $\Q_p$ and let $\rho : G_K \rightarrow \GL_d(k)$ be a \mbox{de Rham} representation. 
There is an isomorphism $D_{\dR}(\ad(\rho)) \cong \ad(D_{\dR}(\rho))$ of filtered $K \otimes_{\Q_p} k$-modules.
\end{lem}

\begin{proof}
Indeed, from $\ad(\rho) \cong \End_k(V_\rho)$ and the fact that $\rho$ and $\ad(\rho)$ are \mbox{de Rham}, we have isomorphisms of filtered $K \otimes_{\Q_p}k$-modules
\begin{align*}
	D_{\dR}(\ad(\rho)) & \cong (B_{\dR} \otimes_{\Q_p} \End_k(V_{\rho}))^{G_K}\\
	& \cong (\End_{(B_{\dR}\otimes_{\Q_p}k)}(B_{\dR}\otimes_{\Q_p} V_{\rho}))^{G_K} \\
	& \cong (\End_{(B_{\dR} \otimes_{\Q_p} k)}(B_{\dR}\otimes_K D_{\dR}(\rho)))^{G_K} \\
	& \cong (B_{\dR}\otimes_K \End_{K\otimes k}(D_{\dR}(\rho)))^{G_K} \\
	& = B_{\dR}^{G_K} \otimes_K \ad(D_{\dR}(\rho))\\
	& = \ad(D_{\dR}(\rho)).\qedhere
	\end{align*}
\end{proof}

\begin{thm}\label{thm:smoothlp}
A closed point $x$ of $\Spec R_{\rhobar}^\square (\tau,\mathbf{v}) [1/p]$
is formally smooth if and only if
$\WD(\rho_x)$ is generic.
\end{thm}

\begin{proof}
Let $k$ denote the residue field of $x$. By \ref{thm:pdefring}, $R_{\rhobar}^\square(\tau,\mathbf{v})_x^\wedge$ has dimension $d^2 + \dim_E \ad(D)/\ad(D)^+.$
Since $\rho_x$ has $p$-adic Hodge type $\mathbf{v}$, this is equal to $d^2 + \dim_k \ad(D_{\dR}(\rho_x))/\ad(D_{\dR}(\rho_x))^+$, which equals  $d^2 + \dim_k D_{\dR}(\ad(\rho_x))/D_{\dR}(\ad(\rho_x))^+$ by \ref{thm:adlem}.

We now analyze the dimension of the tangent space of $R_{\rhobar}^\square(\tau,\mathbf{v})_x^\wedge$. 
By \ref{thm:localBK}, the tangent space of $R_{\rhobar}^\square(\tau,\mathbf{v})_x^\wedge$ has dimension
	\[ \dim_k Z_g^1(K,\ad(\rho_x)) = d^2 + \dim_k H_g^1(K,\ad(\rho_x)) - \dim_k H^0(K,\ad(\rho_x)). \]
So, $R_{\rhobar}^\square(\tau,\mathbf{v})_x^\wedge$ is smooth if and only if
	\[ \dim_k H_g^1(K,\ad(\rho_x)) - \dim_k H^0(K,\ad(\rho_x)) = \dim_k D_{\dR}(\ad(\rho_x))/D_{\dR}(\ad(\rho_x))^+,
	\]
equivalently,
	\begin{equation}\label{eqn:dRtangent}
	\dim_{\Q_p} H_g^1(K,\ad(\rho_x)) - \dim_{\Q_p} H^0(K,\ad(\rho_x)) = \dim_{\Q_p} D_{\dR}(\ad(\rho_x))/D_{\dR}(\ad(\rho_x))^+.
	\end{equation}
Before proceeding, we introduce some notation. 
If $W$ is a finite dimensional $\Q_p$-vector space with a continuous $\Q_p$-linear $G_K$-action, define the $\Q_p$-vector spaces as in \cite{BlochKato}*{\S3}:
	\begin{align*}
	H_e^1(K, W) &:= \ker(H^1(K, W) \rightarrow H^1(K, B_{\cris}^{\varphi = 1} \otimes_{\Q_p} W)), \\
	H_f^1(K, W) &:= \ker(H^1(K, W) \rightarrow H^1(K, B_{\cris} \otimes_{\Q_p} W)).
	\end{align*}
The pairing $(X,Y) \mapsto \tr_{k/\Q_p}(\tr(XY))$ is perfect on $\ad(\rho_x)$, so induces an isomorphism 	
	\[\ad(\rho_x)(1) \cong \Hom_{\Q_p}(\ad(\rho_x),\Q_p(1)). \] 
Then, by \cite{BlochKato}*{Proposition~3.8},
	\begin{align*} 
	& \dim_{\Q_p} H_g^1(K,\ad(\rho_x)) - \dim_{\Q_p} H^0(K,\ad(\rho_x)) \\
	\quad {}	
	& = 
	\dim_{\Q_p} H^1(K, \ad(\rho_x)) - \dim_{\Q_p} H_e^1(K,\ad(\rho_x)(1)) - \dim_{\Q_p} H^0(K,\ad(\rho_x)) \\ 
	\quad {} 
	& = \dim_{\Q_p} H_f^1(K,\ad(\rho_x)) + \dim_{\Q_p} H_f^1(K, \ad(\rho_x)(1)) - \dim_{\Q_p} H_e^1(K,\ad(\rho_x)(1)) - \dim_{\Q_p} H^0(K,\ad(\rho_x)).
	\end{align*}
Using \cite{BlochKato}*{Corollary~3.8.4}, this last expression equals
	\[ \dim_{\Q_p} D_{\dR}(\ad(\rho_x))/D_{\dR}(\ad(\rho_x))^+ + \dim_{\Q_p} D_{\cris}(\ad(\rho_x)(1))^{\varphi = 1}. \]
Plugging this into \eqref{eqn:dRtangent}, we see that $R_{\rhobar}^\square(\tau,\mathbf{v})_x^\wedge$ is formally smooth if and only if $D_{\cris}(\ad(\rho_x)(1))^{\varphi = 1} = 0$. This happens if and only if $\WD(\rho_x)$ is generic by part \ref{lpgeneric:crys} of \ref{thm:lpgeneric}.
\end{proof}

More thorough investigations of the smooth and singular loci in the case $d=2$ and the case $d=3$ and $K_0 = \Q_p$ are carried out in \cite{KisinFM}*{(A.1)} and \cite{BellovinSmooth}*{\S7}. 
In particular, when $d=2$ Kisin shows in \cite{KisinFM}*{(A.1)} that $\Spec R_{\rhobar}^\square(\tau,\mathbf{v})[1/p]$ is reduced, and is either smooth or is the union of $2$ smooth closed subspaces. 
When $d = 3$ and $K_0 = \Q_p$, Bellovin shows in \cite{BellovinSmooth}*{\S7.3} that $\Spec R_{\rhobar}^\square(1,\mathbf{v})[1/p]$ is the union of $3$ closed subspaces, two of which are smooth, and one of which is singular.

In practice it is often important to know that given a representation $\rho$ of $G_K$, the restriction $\rho|_{G_L}$ defines a smooth point for any finite extension $L/K$. 
For example, $\rho = \epsilon\oplus \chi$ with $\chi$ a nontrivial finite order character defines a smooth point on the corresponding potentially semistable deformation ring, but the restriction $\rho|_{G_L}$ to any $G_L$ that trivializes $\chi$ will not. 
It is not hard to see that any (mixed) pure Galois representation (i.e. one that satisfies the conclusion of the Weight--Monodromy Conjecture) will define a smooth point after any finite base change. 
A similar sufficient condition was noticed by Calegari \cite{CalegariEven2}*{Lemma~2.6}.
However, as the following example illustrates, the condition that $\WD(\rho|_{G_L})$ is generic for any finite extension $L/K$ is strictly weaker than either of these.
%
\begin{eg}\label{eg} 
Choose a cocycle $\kappa$ of $G_{\Q_p}$ valued in $\Q_p(1)$ such that the cohomology class of $\kappa$ does not lie in $H_f^1(\Q_p,\Q_p(1))$. 
Then the representation
	\[ \rho = \begin{pmatrix} 1 & & \\ & 1 & \kappa  \\ & & \epsilon^{-1} \end{pmatrix} \]
is semistable noncrystalline, and if $L/\Q_p$ is any finite extension, 
the Weil--Deligne representation $\WD(\rho|_{G_L}) = (r,N)$ is given by
	\[ r(\Frob_L) = \begin{pmatrix} 1 & &  \\ & 1 &  \\ & & p^f  \end{pmatrix} \quad \text{and} \quad
	N = \begin{pmatrix}  & 0 &  \\ &  & 1  \\ &  &   \end{pmatrix}, \]
where $f$ denotes the residue degree of $L$. 
A straightforward check shows that $\WD(\rho|_{G_L})$ is generic, so letting $\mathbf{v}$ denote the $p$-adic Hodge type of $\rho$, the restriction $\rho|_{G_L}$ defines a smooth point on $\Spec R_{\rhobar|_{G_L}}^\square(1,\mathbf{v}_L)[1/p]$ for any finite extension $L/\Q_p$ (where $\mathbf{v}_L := L\otimes_K \mathbf{v}$). 
Some related and more detailed computations are carried out in \cite{BellovinSmooth}*{\S7.3}.
\end{eg}

\begin{rmk}\label{rmk:gfsame}
Let $\rho : G_K \rightarrow \GL_d(E)$ a potentially semistable representation. 
We saw in the proof of \ref{thm:smoothlp} that \cite{BlochKato}*{Proposition~3.8} and \cite{BlochKato}*{Corollary~3.8.4} imply
	\[ \dim_{\Q_p} H_g^1(K,\ad(\rho))  = \dim_{\Q_p} H_f^1(K,\ad(\rho)) + \dim_{\Q_p} D_{\cris}(\ad(\rho)(1))^{\varphi = 1}. \]
So part~\ref{lpgeneric:crys} of \ref{thm:lpgeneric} shows $H_f^1(K,\ad(\rho)) = H_g^1(K,\ad(\rho))$ if and only if $\WD(\rho)$ is generic. 

We can generalize one direction of this slightly. 
Let $W$ be a representation of $G_K$ on a $d^2$-dimensional $E$-vector space such that it's restriction to $G_L$, for some $L/K$ finite, is isomorphic to $\ad(\rho)$ with $\rho : G_L \rightarrow \GL_d(E)$ a potentially semistable representation. 
Then $H_f^1(K,W) = H_g^1(K,W)$ if $\WD(\rho)$ is generic. 
Indeed, the commutative diagram
	\[ \begin{tikzcd} 
	H^1(K,W) \arrow{d} \arrow{r} & H^1(K,B_{\cris}\otimes_{\Q_p}W) \arrow{d} \arrow{r} & H^1(K,B_{\dR}\otimes_{\Q_p} W) \arrow{d} \\
	H^1(L,W)  \arrow{r} & H^1(L,B_{\cris}\otimes_{\Q_p}W)  \arrow{r} & H^1(L,B_{\dR}\otimes_{\Q_p} W)
	\end{tikzcd} \]
has injective vertical arrows by restriction-corestriction. 
It easily follows that $H_f^1(K,W) = H_g^1(K,W)$ if $H_f^1(L,W) = H_g^1(L,W)$. 
We mention this slight generalization because in our global applications we do not wish to restrict ourselves to CM extensions $F/F^+$ such that every $v|p$ in $F^+$ splits in $F$.
\end{rmk}
%
%
%

\subsection{Global Galois deformation rings}\label{sec:global}

Throughout this subsection we assume $p>2$.

We recall the Clozel--Harris--Taylor group scheme $\mathcal{G}_d$, which is the group scheme over $\Z$ defined as the semidirect product
	\[ (\GL_d \times \GL_1) \rtimes \{1,\jmath\} = \calG_d^0 \rtimes \{1,\jmath\}, \]
where $\jmath (g,a) \jmath = (a \transp g^{-1},a)$, and the homomorphism $\nu : \mathcal{G}_d \rightarrow \GL_1$ given by $\nu(g,a) = a$ and $\nu(\jmath) = -1$. 
We let $\frakgl_d = \Lie \GL_d \subset \Lie \mathcal{G}_d$, and let $\ad$ denote the adjoint action of $\calG_d$ on $\frakgl_d$, i.e 
	\[ \ad(g,a)(x) = gxg^{-1} \quad \text{and} \quad \ad(\jmath)(x) = -{}^t x.\]
If $\Gamma$ is a group, $A$ is a commutative ring and 
	\[ r : \Gamma \lra \calG_d(A) \]
is a homomorphism, we write $\ad(r)$ for $\frakgl_d(A)$ with the adjoint action $\ad\circ r$ of $\Gamma$.

The following is (part of) \cite{CHT}*{Lemma~2.1.1}.

\begin{lem}\label{thm:Gdhoms}
Let $\Gamma$ be a topological group with an open subgroup $\Delta$ of index $2$. Fix some $\gamma_0 \in \Gamma \smallsetminus \Delta$. Let $A$ be a topological ring. There is a natural bijection between the following two sets.
\begin{enumerate}
	\item Continuous homomorphisms $r : \Gamma \rightarrow \calG_d(A)$ inducing an isomorphism $\Gamma/\Delta \xrightarrow{\sim} \calG_d(A)/\calG_d^0(A)$.
	\item Triples $(\rho,\mu,\langle\cdot,\cdot\rangle)$, where $\rho : \Delta \rightarrow \GL_d(A)$ and $\mu : \Gamma \rightarrow A^\times$ are continuous homomorphisms and $\langle \cdot, \cdot \rangle$ is a perfect $A$-linear pairing on $A^d$ satisfying
		\[ \langle \rho(\delta) a, \rho(\gamma_0\delta\gamma_0^{-1}) b \rangle = \mu(\delta)\langle a, b \rangle \quad \text{and}
		\quad \langle a, \rho(\gamma_0^2) b \rangle = - \mu(\gamma_0) \langle b, a \rangle \]
	for all $a,b\in A^d$ and $\delta \in \Delta$.
\end{enumerate}
Under this bijection, $\mu(\gamma) = (\nu\circ r)(\gamma)$ for all $\gamma \in \Gamma$, and $\langle a, b \rangle = {}^t a P^{-1} b$ for $r(\gamma_0) = (P,-\mu(\gamma_0))\jmath$.
\end{lem}

If $\Gamma$ is a group, $A$ is a commutative ring, $r : \Gamma \rightarrow \calG_d(A)$ is a homomorphism, and $B$ is an $A$-algebra, we will write $r \otimes_A B$ for the composite of $r$ with the map $\calG_d(A) \rightarrow \calG_d(B)$. 
If $[r]$ is a $1+\M_d(\frakm_A)$-conjugacy class of such homomorphisms, we will write $[r]\otimes_A B$ for the $1+\M_d(\frakm_B)$-conjugacy class $[r\otimes_A B]$. 

If $\Gamma$ is a group, $A$ is a commutative ring, $r : \Gamma \rightarrow \calG_d(A)$ is a homomorphism, and $\Delta$ is a subgroup of $\Gamma$ such that $r(\Delta) \subseteq \calG_d^0(A)$, we will write $r\rest_\Delta$ for the composite of the restriction of $r$ to $\Delta$ with the projection $\calG_d^0(A) \rightarrow \GL_d(A)$. 
In particular, we view $A^d$ as an $A[\Delta]$-module via $r\rest_{\Delta}$.

We recall \cite{CHT}*{Definition~2.1.6}:

\begin{defn}\label{def:schur}
Let $\Gamma$ be a group with index two subgroup $\Delta$. Fix $\gamma_0 \in \Gamma\smallsetminus \Delta$. Let $k$ be a field and let $r : \Gamma \rightarrow \calG_d(k)$ be a homomorphism with $\Delta = r^{-1}(\calG_d^0(k))$. We say that $r$ is \emph{Schur} if all $\Delta$-irreducible subquotients of $k^n$ are absolutely irreducible and for all $\Delta$-invariant subspaces $k^n \supset W_1 \supset W_2$ such that $k^n/W_1$ and $W_2$ are irreducible, we have $(k^n/W_1)^{\gamma_0} \not\cong W_2^\vee \otimes(\nu\circ r)$.
\end{defn} 

\noindent Note that if $r\rest_\Delta$ is absolutely irreducible, then $r$ is Schur. 

\subsubsection{}\label{sec:globalH}
Before proceeding with deformation theory, we prove some results on the cohomology of the adjoint representation valued in a finite extension of $\Q_p$. 
Let $F$ be a CM field with maximal totally real subfield $F^+$. 
Let $S$ be a finite set of finite places of $F^+$ containing all those above $p$. 
Let $k$ be a finite extension of $\Q_p$, let
	\[ r : \Gal(F(S)/F^+) \lra \calG_d(k) \]
be a continuous homomorphism inducing an isomorphism $\Gal(F/F^+) \xrightarrow{\sim} \calG_d(k)/\calG_d^0(k)$, and let $\mu = \nu\circ r$. 
For each $v|\infty$ in $F$, let $c_v \in G_{F^+}$ be a choice of complex conjugation. 
Recall we have assumed $p>2$. 

\begin{lem}\label{thm:EPchar}
Let the notation and assumptions be as in \ref{sec:globalH} above. 
Then
	\[ \sum_{i=0}^2 (-1)^i\dim_k H^i(F(S)/F^+,\ad(r)) = - d^2[F^+:\Q] + \sum_{v|\infty} \frac{d(d+\mu(c_v))}{2}.\]
\end{lem}

\begin{proof}
An easy computation (see \cite{CHT}*{Lemma~2.1.3}) shows $\dim_k H^0(F_v^+,\ad(r)) = \frac{d(d+\mu(c_v))}{2}$ for each $v|\infty$. 
Using \cite{CHT}*{Lemma~2.1.5}, we may assume $r$ takes values in $\calG_d(\calO_k)$, where $\calO_k$ is the ring of integers of $k$. 
The lemma now follows from \cite{CHT}*{Lemma~2.3.3} by an argument as in \cite{KisinOverConvFM}*{Lemma~9.7}.
\end{proof}

\begin{lem}\label{thm:cohom}
Let the assumptions and notation be as in \ref{sec:globalH} above. 
Assume further that $r\rest_{G_w}$ is \mbox{de Rham} with regular $p$-adic Hodge type for every $w|p$ in $F$, that $\mu(c_v) = -1$ for every $v|\infty$, and that $\ad(r)^{G_{F^+}} = 0$. 
If $H_g^1(F(S)/F^+,\ad(r)) = 0$, then the following hold. 
\begin{enumerate}
\item\label{cohom:H1dim} $\dim_k H^1(F(S)/F^+,\ad(r)) = \frac{d(d+1)}{2}[F^+:\Q]$.
\item\label{cohom:H2} $H^2(F(S)/F^+,\ad(r)) = 0$.
\item\label{cohom:H1iso} The natural map
	\[ H^1(F(S)/F^+,\ad(r)) \lra \prod_{v|p} H^1(F_v^+,\ad(r))/H_g^1(F_v^+,\ad(r)) \]
is an isomorphism.
\end{enumerate}
\end{lem}

\begin{proof}
The argument is exactly as in the proof of \cite{KisinGeoDefs}*{Theorem~8.2}. 
We give the details. 

Using our assumption that $\mu(c_v) = -1$ for every $v|\infty$, and that $\ad(r)^{G_{F^+}} = 0$, \ref{thm:EPchar} implies 
	\[ \dim_k H^1(F(S)/F^+,\ad(r)) - \dim_k H^2(F(S)/F^+,\ad(r)) =  \frac{d(d+1)}{2}[F^+:\Q], \]
so parts~\ref{cohom:H1dim} and \ref{cohom:H2} are equivalent. 
Using the assumption $H_g^1(F(S)/F^+,\ad(r)) = 0$, there is an injection
	\[ H^1(F(S)/F^+,\ad(r)) \lra \prod_{v|p} H^1(F_v^+,\ad(r))/H_g^1(F_v^+,\ad(r)), \]
and all three parts of the lemma will follow from showing 
	\[ \dim_k H^1(F_v^+,\ad(r))/H_g^1(F_v^+,\ad(r)) \le \frac{d(d+1)}{2}[F_v^+:\Q_p] \]
for each $v|p$ in $F^+$. 
Equivalently, it suffices to show that for each $v|p$ in $F^+$, the image of 
	\[ H^1(F_v^+,\ad(r)) \lra H^1(F_v^+,B_{\dR} \otimes_{\Q_p} \ad(r)) \]
has $k$-dimension $\le \frac{d(d+1)}{2}[F_v^+:\Q_p]$. 
The $k$-vector space $H^1(F_v^+,B_{\dR}\otimes_{\Q_p} \ad(r))$ has a filtration induced by the filtration on $B_{\dR}$,  
and the image of $H^1(F_v^+,\ad(r))$ is contained in the $\Fil^0$. 
It thus suffices to show that for each $v|p$ in $F^+$, 
	\[ \dim_k \Fil^0 H^1(F_v^+,B_{\dR}\otimes_{\Q_p} \ad(r)) \le \frac{d(d+1)}{2}[F_v^+:\Q_p]. \]
From the filtered $G_v$-equivariant isomorphism $B_{\dR} \otimes_{\Q_p} \ad(r) \cong B_{\dR} \otimes_{F_v^+} D_{\dR}(\ad(r))$, and the fact that $H^1(F_v^+,B_{\dR}) \cong F_v^+$ (see \cite{Tatepdiv}*{\S3}), we have filtered isomorphisms
	\begin{align*} H^1(F_v^+,B_{\dR} \otimes_{\Q_p}\ad(r)) & \cong H^1(F_v^+,B_{\dR} \otimes_{F_v^+} D_{\dR}(\ad(r))) \\
	& \cong H^1(F_v^+, B_{\dR}) \otimes_{F_v^+} D_{\dR}(\ad(r)) \\
	& \cong D_{\dR}(\ad(r)). \end{align*}
So, we are reduced to showing that $\dim_k \Fil^0 D_{\dR}(\ad(r)) = \frac{d(d+1)}{2}[F_v^+:\Q_p]$ for each $v|p$ in $F^+$. 
For any $w|p$ in $F$, since $r\rest_{G_w}$ is \mbox{de Rham}, there is an isomorphism of filtered $F_w \otimes_{\Q_p} k$-modules (see \ref{thm:adlem})
	\[(B_{\dR}\otimes_{\Q_p} \ad(r))^{G_w} \cong \ad((B_{\dR}\otimes_{\Q_p} V_{r\rest_{G_w}})^{G_w}),\]
and since $r\rest_{G_w}$ has regular $p$-adic Hodge type,
	\[ \dim_k \Fil^0 \ad((B_{\dR}\otimes_{\Q_p} V_{r\rest_{G_w}})^{G_w}) = \frac{d(d+1)}{2}[F_w:\Q_p]. \]
If $v$ splits in $F$ as $w w^c$, then the choice of $w$ induces an isomorphism $F_v^+ \cong F_w$, and
	\[ \dim_k \Fil^0 D_{\dR}(\ad(r)) = \frac{d(d+1)}{2}[F_v^+:\Q_p]. \]
If $v$ does not split in $F$, then letting $w$ denote the unique place dividing $v$ in $F$, we have $[F_w:F_v^+] = 2$, and there is a filtered isomorphism
	\[ (B_{\dR} \otimes_{\Q_p} \ad(r))^{G_w} \cong F_w \otimes_{F_v^+} D_{\dR}(\ad(r)), \]
and we have
	\[  \dim_k \Fil^0 D_{\dR}(\ad(r)) = \frac{1}{2} \dim_k \Fil^0(B_{\dR} \otimes_{\Q_p} \ad(r))^{G_w} = \frac{d(d+1)}{2}[F_v^+:\Q_p]. \qedhere \]
\end{proof}

We now recall the $\calG_d$-valued deformation theory of \cite{CHT}.

\begin{defn}\label{def:lift}
Let $k$ be either a finite extension of $\F_p$ or of $\Q_p$. 
Let $\Gamma$ be a topological group and let $\rbar : \Gamma \rightarrow \calG_d(k)$ be a continuous homomorphism. 
Let $A$ be a pro-Artinian local ring with a fixed isomorphism $A/\frakm_A \xrightarrow{\sim} k$. 

A \emph{lift} of $\rbar$ to $A$ is a homomorphism $r \rightarrow \calG_d(A)$ such that $r \otimes_A k = \rbar$ and such that for any Artinian quotient $A \rightarrow A'$, the homomorphism $r\otimes_A A'$ is continuous, where we give $A'$ the discrete topology if $k$ is a finite extension of $\F_p$, and the topology as a finite dimensional $k$-vector space if $k$ is a finite extension of $\Q_p$.
A \emph{deformation} of $\rbar$ to $A$ is a $1+\M_d(\frakm_A)$-conjugacy class of lifts.

For a finite set $T$, a $T$-\emph{framed lift} of $\rbar$ to $A$ is a tuple $(r,\{\alpha_v\}_{v\in T})$ where $r$ is a lift of $\rbar$ to $A$ and $\alpha_v \in \ker(\GL_d(A)\rightarrow \GL_d(k))$. 
We say two $T$-framed lifts $(r,\{\alpha_v\}_{v\in  T})$ and $(r',\{\alpha_v'\}_{v\in T})$ to $A$ are equivalent if there is $g\in \ker(\GL_d(A)\rightarrow\GL_d(k))$ such that $g r g^{-1} = r'$ and $g\alpha_v = \alpha_v'$ for each $v\in T$. 
A $T$-\emph{framed deformation} of $\rbar$ to $A$ is an equivalence class of $T$-framed lifts.
\end{defn}

If $r$ is a lift, we will write $[r]$ for the corresponding deformation. If $(r,\{\alpha_v\})$ is a $T$-framed lift, we will write $[r,\{\alpha_v\}]$ for the corresponding $T$-framed deformation.


We will introduce a slight variation of the global deformation problem of \cite{CHT}*{\S2.3}.



\begin{defn}\label{def:globaldatum}
A \emph{global deformation datum} is a tuple
	\[ \calS = (F/F^+, S, \tildeS, \calO, \rbar, \mu, \{R_w\}_{w\in \tildeS}) \]
where
	\begin{itemize}
	\item $F$ is a CM field with maximal totally real subfield $F^+$;
	\item $S$ is a finite set of finite places of $F^+$; 
	\item $\tildeS$ is a finite set of finite places of $F$ such that every $w\in \tildeS$ is split over some $v\in S$, and $\tildeS$ contains at most one place above any $v\in S$;
	\item $\calO$ is the ring of integers of some finite extension of $\Q_p$ with residue field $\F$;
	\item $\rbar: \Gal(F(S)/F^+) \rightarrow \calG_d(\F)$ is a continuous homomorphism;
	\item $\mu: \Gal(F(S)/F^+) \rightarrow \calO^\times$ is a continuous character with $\mu \bmod {\frakm_{\calO}} = \nu\circ\rbar$;
	\item for each $w\in \tildeS$, $R_w$ is a quotient of $R_w^\square : = R_{\rbar\rest_{G_w}}^\square$ satisfying the following property: 
	if $\rho : G_w \rightarrow \GL_d(A)$ is a lift of $\rbar\rest_{G_w}$ to $A$ and $g\in 
	1+\M_d(\frakm_A)$,
	then the map $R_w^\square \rightarrow A$ induced by $\rho$ factors through $R_w$ if and only if the map $R_w^\square \rightarrow A$ induced by $g\rho g^{-1}$ factors through $R_w$.
	\end{itemize}
\end{defn}

This differs from the definition in \cite{CHT}*{\S2.3} in that our ramification set $S$ may contain places that do not split in $F/F^+$, and $\tildeS$ is not required to contain a place above every $v\in S$. 
When proving modularity of Galois representations, one can use base change and descent to reduce to the case that the ramification set splits in $F/F^+$, and for the proof of \ref{thm:BK}, which implies Theorems~\ref{thm:BKA}, \ref{thm:thmB}, and \ref{thm:thmC} from the introduction, it would also suffice to consider this situation because we may also use base change in its proof. 
But we wish to have the statement of Theorem~\ref{thm:thmC} in the above level of generality for other applications where it is not obvious (at least not to the author) how to apply base change and descent. 
One such application is to the density of automorphic points in deformation rings. 
B\"{o}ckle's strategy \cite{BockleDensity} for proving density of modular points in universal deformation rings is to show that every irreducible component of the universal deformation ring contains a smooth modular point, and then to use the infinite fern of Gouv\^{e}a and Mazur starting at such a smooth point to ``fill out" the irreducible component. 
Chenevier \cite{ChenevierFern} has constructed an infinite fern in the $3$-dimensional conjugate self-dual case, assuming $p$ is totally split in the CM field $F$. 
If one knows a priori that automorphic points always define smooth points on the universal deformation ring, then to prove new cases of the density of automorphic points in $3$-dimensional conjugate self-dual deformation rings, one now just has to show that every irreducible component contains an automorphic point. 
These ideas will be developed further in forthcoming work of the author, and as the results in \cite{ChenevierFern} make no assumption on the splitting behaviour in $F$ of the places in $S \smallsetminus \{v|p\}$, we also wish to make no such assumption.

\begin{defn}\label{def:globaldef} Let $\calS=(F/F^+,S,\tildeS,\calO,\rbar,\mu,\{R_w\}_{w\in \widetilde{S}})$ be a global deformation datum, and let $A$ be a $\CNL_{\calO}$-algebra. We say a lift $r : G_{F^+} \rightarrow \calG_d(A)$ of $\rbar$ to $A$ is \emph{type} $\calS$ if
	\begin{itemize}
	\item $r$ factors through $\Gal(F(S)/F^+)$;
	\item $\nu \circ r = \mu$;
	\item for each $w\in \tildeS$, the $\CNL_{\calO}$-morphism $R_w^\square \rightarrow A$ induced by the lift $r\rest_{G_w}$ of $\rbar\rest_{G_w}$, factors through $R_w$.
	\end{itemize}
We say a deformation of $\rbar$ to $A$ is \emph{type} $\calS$ if one (equivalently any) lift in its deformation class is type $\calS$. 
We let $D_{\calS}$ be the set valued functor on $\CNL_{\calO}$ that takes a $\CNL_{\calO}$-algebra $A$ to the set of deformations of type $\calS$. 
If $D_{\calS}$ is representable, we call the representing object the \emph{universal type} $\calS$ \emph{deformation ring} and denote it by $R_{\calS}$.

For any $T\subseteq \tildeS$, we say a $T$-framed deformation $[r,\{\alpha_w\}]$ of $\rbar$ to $A$ is \emph{type} $\calS$ if $[r]$ is a type $\calS$ deformation of $\rbar$. 
We let $D_{\calS}^{\square_T}$ be the set valued functor on $\CNL_{\calO}$ that takes a $\CNL_{\calO}$-algebra $A$ to the set of $T$-framed deformations of type $\calS$. 
If $D_{\calS}^{\square_T}$ is representable, we call the representing object the \emph{universal type} $\calS$ $T$-\emph{framed deformation ring} and denote it by $R_{\calS}^{\square_T}$. 
If $T = \tildeS$, then we will write $D_{\calS}^\square$ and $R_{\calS}^\square$ for $D_{\calS}^{\square_T}$ and $R_{\calS}^{\square_T}$, respectively, and call $R_{\calS}^\square$ the \emph{universal type} $\calS$ \emph{framed deformation ring}.
\end{defn}

If $\calS=(F/F^+,S,\tildeS,\calO,\rbar,\mu,\{R_w\}_{w\in \tildeS})$ is a global deformation datum and $T\subseteq \tildeS$, we set
	\[ R_T^\square = \widehat{\otimes}_{w\in T} R_w^\square \quad \text{and} \quad
	R_{\calS,T}^{\loc} = \widehat{\otimes}_{w\in T} R_w \] 
Note that $R_{\calS,T}^{\loc}$ is naturally a quotient of $R_T^\square$. 
If $T = \tildeS$, then we will write $R_{\calS}^{\loc}$ for $R_{\calS,T}^{\loc}$. 

The following proposition follows from \cite{CHT}*{Proposition~2.2.9}.

\begin{prop}\label{thm:globalrep}
Let $\calS=(F/F^+,S,\tildeS,\calO,\rbar,\mu,\{R_w\}_{w\in \tildeS})$ be a global deformation datum, and let $T\subseteq \tildeS$. 
Assume $\rbar$ is Schur.

The functors $D_{\calS}^{\square_T}$ and $D_{\calS}$ are representable. 
There is a canonical $\CNL_{\calO}$-morphism $R_{\calS,T}^{\loc} \rightarrow R_{\calS}^{\square_T}$. 
There is a canonical $\CNL_{\calO}$-morphism $R_{\calS} \rightarrow R_{\calS}^{\square_T}$, and a choice of lift 
	\[ r_{\calS}^{\univ} : \Gal(F(S)/F^+) \lra \calG_d(R_{\calS}) \]
for the universal type $\calS$ deformation $[r_{\calS}^{\univ}]$ determines an extension of this $\CNL_{\calO}$-morphism to an isomorphism $R_{\calS}[[X_1,\ldots,X_{d^2\abs{T}}]] \stackrel{\sim}{\lra} R_{\calS}^{\square_T}$.

%
\end{prop}

\subsubsection{}\label{sec:globalcomplete}
For the remainder of this section, we fix a global deformation datum 
	\[ \calS = (F/F^+,S,\tildeS,\calO,\rbar,\mu,\{R_w\}_{w\in \tildeS}), \]
with $\rbar$ Schur, where we will specify the rings $R_w$ in the statements of the following propositions. 
Let $R_{\calS}$ be the universal type $\calS$ deformation ring and let $x$ be a closed point of $\Spec R_{\calS}[1/p]$ with residue field $k$. 
Let $\calO_k$ be the ring of integers of $k$, and let $r$ be a type $\calS$ lift of $\rbar$ to $\calO_k$ such that the map $R_{\calS} \rightarrow \calO_k$ induced by $[r]$ induces $x: R_{\calS}[1/p] \rightarrow k$. 
Set $r_x = r \otimes_{\calO_k} k$.

\begin{prop}\label{thm:bigglobalcomplete} 
Let the notation and assumptions be as in \ref{sec:globalcomplete} above, with $R_w = R_w^\square$ for every $w\in \tildeS$.
\begin{enumerate}
\item\label{bigglobalcomplete:rep} Let $D_{\calS,r_x}$ be the functor on $\Ar_k$ that sends an $\Ar_k$-algebra $B$ to the set of deformations $[r_B]$ of $r_x$ to $B$ that factor through $\Gal(F(S)/F^+)$ and satisfy $\nu\circ r_B = \mu$. 
The functor $D_{\calS,r_x}$ is prorepresented by $(R_{\calS})_x^\wedge$.
\item\label{bigglobalcomplete:tangent}  The tangent space of $R_{\calS}[1/p]$ at $x$ is canonically isomorphic to $H^1(F(S)/F^+,\ad(r_x))$.
\item\label{bigglobalcomplete:dim} Let $S_\infty$ be the set of infinite places in $F^+$, and for every $v\in S_\infty$, let $c_v$ be a choice of complex conjugation at $v$. 
Then $(R_{\calS})_x^\wedge$ is isomorphic to a power series over $k$ in $g = \dim_k H^1(F(S)/F^+,\ad(r_x))$ variables modulo $r$ relations with $r \le \dim_k H^2(F(S)/F^+,\ad(r_x))$, and 
	\[ g - r \ge d^2[F^+:\Q] - \sum_{v \in S_\infty} \frac{d(d+\mu(c_v))}{2}. \]
\end{enumerate}
\end{prop}

\begin{proof}
The proof of \ref{bigglobalcomplete:rep} is almost identical to that of \cite{KisinFinFlat}*{Proposition~2.3.5}. We sketch the details. 
We will not use the language of groupoids here, but the results we will reference from \cite{KisinFinFlat} stated in terms of groupoids imply our results stated in terms of functors by \cite{KisinFinFlat}*{\S A.5}. 

Since $R_w^\square$, for $w\in \tildeS$, imposes no condition on our lifts, it is easy to see that if $\calS'$ is the deformation datum
 	\[\calS' = (F/F^+,S,\emptyset,\calO,\rbar,\mu, \emptyset), \]
then there is a canonical isomorphism $R_{\calS} \cong R_{\calS'}$. 
For the remainder of the proof, we assume $\tildeS = \emptyset$.  

For any $\Ar_k$-algebra $B$, we let $\mathrm{Int}(B)$ be the set of all finite $\calO$-subalgebras $A \subset B$ such that $A[1/p] = B$. 
Note that any $A \in \mathrm{Int}(B)$ comes equipped with a canonical $\calO$-algebra map $A \rightarrow \calO_k$ via $A \subset B \rightarrow B/\frakm_B = k$. 
Also note that $\mathrm{Int}(B)$ is naturally filtered. 
We let $D_{\calS,(r)}$ be the set valued functor on $\Ar_k$ defined by
	\[ D_{\calS,(r)}(B) = \varinjlim_{A\in \mathrm{Int}(B)} \{ [r_A] \text{ is a type } \calS \text{ deformation of } \rbar \text{ such that } [r_A]\otimes_A \calO_k = [r] \}. \]
By \cite{KisinFinFlat}*{Lemma~2.3.3}, the functor $D_{\calS,(r)}$ is prorepresented by $(R_{\calS})_x^\wedge$. 
There is a natural morphism of functors $D_{\calS,(r)} \rightarrow D_{\calS, r_x}$ that we wish to show is an isomorphism. 
For any $\Ar_k$-algebra $B$, and continuous homomorphism $r_B : \Gal(F(S)/F^+) \rightarrow \calG_d(B)$, an argument as in \cite{KisinOverConvFM}*{Proposition~9.5} shows that there is some $A\in \mathrm{Int}(B)$ and a continuous homomorphism $r_A : \Gal(F(S)/F^+) \rightarrow \calG_d(A)$ such that $r_B = r_A \otimes_A B$. 
If $r_B \in D_{\calS,r_x}(A)$, then such an $r_A$ must by type $\calS$, so $D_{\calS,(r)} \rightarrow D_{\calS,r_E}$ is surjective. 
To see that it is injective, note that for any $\Ar_k$-algebra $B$ and any $g\in 1+\M_d(\frakm_B)$, the $\calO$-algebra generated by the entries of $g$ is finite over $\calO$. 
Thus, any two lifts of $r_x$ to an $\Ar_k$-algebra $B$ defining the same deformation arise from lifts to some $A\in \mathrm{Int}(B)$ that define the same deformation to $A$. 
This finishes the proof of part~\ref{bigglobalcomplete:rep}.

For part~\ref{bigglobalcomplete:tangent}, letting $Z^1(F(S)/F^+,\ad(r_x))$ be the space of continuous $1$-cocycles of $\Gal(F(S)/F^+)$ with values in $\ad(r_x)$, the map $\kappa \mapsto r_\kappa: = (1+\varepsilon \kappa)r_x$ defines an isomorphism from $Z^1(F(S)/F^+,\ad(r_x))$ to the $k$-vector space of lifts $r_{\varepsilon}$ of $r_x$ to the dual numbers $k[\varepsilon]$ that factor through $\Gal(F(S)/F^+)$ and satisfy $\nu \circ r_{\varepsilon} = \mu$. 
Two such cocycles $\kappa$ and $\kappa'$ are cohomologous if and only if $r_\kappa$ and $r_{\kappa'}$ are conjugate by an element of $1+\varepsilon\M_d(k)$, so this map induces a canonical isomorphism from $H^1(F(S)/F^+,\ad(r_x))$ to the $k$-vector space of deformations $[r_{\varepsilon}]$ of $r_k$ to $k[\varepsilon]$ with $\nu \circ r_{\varepsilon} = \mu$, which is isomorphic to the tangent space of $R_{\calS}[1/p]$ at $x$ by part~\ref{bigglobalcomplete:rep}.

We now show part~\ref{bigglobalcomplete:dim}. 
By part~\ref{bigglobalcomplete:tangent}, we can fix a surjection $A: = k[[X_1,\ldots,X_g]] \rightarrow (R_{\calS})_x^\wedge$ with $g = \dim_k H^1(F(S)/F^+,\ad(r_x))$ that induces an isomorphism on tangent spaces. 
Let $J$ denote its kernel. 
Choose a lift $r_x^{\univ} : \Gal(F(S)/F^+) \rightarrow \calG_d((R_{\calS})_x^\wedge)$ of $r_x$ in the universal $(R_{\calS})_x^\wedge$-valued deformation. 
We view $A/\frakm_A J$ and $(R_{\calS})_x^\wedge$ as topological rings with the topology as inverse limits of finite dimensional $k$-vector spaces. 
We can choose a continuous set-theoretic section $s_0 : (R_{\calS})_x^\wedge \rightarrow A/\frakm_A J$ of the surjection $A/\frakm_A J  \rightarrow (R_{\calS})_x^\wedge$ with the property that $s_0(a) = a$ for any $a\in k$. 
Then $s = s_0 \circ r_x^{\univ}$ is a continuous set-theoretic lift $s : \Gal(F(S)/F^+) \rightarrow \calG_d(A/\frakm_A J)$ of $r_x^{\univ}$ with the property that $\nu\circ s(\sigma) = \mu(\sigma)$ for all $\sigma \in \Gal(F(S)/F^+)$. 
We can then define a $2$-cocycle $\kappa$ of $\Gal(F(S)/F^+)$ valued in $\ad(r_x) \otimes_k J/\frakm_A J$ by 
	\begin{align*}(\kappa(\sigma,\tau), 1) = s(\sigma\tau) s(\sigma)^{-1} s(\tau)^{-1} &\in \ker(\calG_d^0(A/\frakm_A J) \rightarrow \calG_d^0((R_{\calS})_x^\wedge) \\
	& \cong (\ad(r_x) \otimes_k J/\frakm_A J) \times (1+J)/(1+\frakm_A J). \end{align*}
The cohomology class
	\[ [\kappa] \in H^2(F(S)/F^+,\ad(r_x)\otimes_k J/\frakm_A J) \cong H^2(F(S)/F^+,\ad(r_x)) \otimes_k J/\frakm_A J \]
does not depend on our choices. 
The argument of \cite{MazurDefGalRep}*{\S1.6} shows that the map $\Hom_k (J/\frakm_A J, k) \rightarrow H^2(F(S)/F^+,\ad(r_x))$ given by $f \mapsto (1\otimes f)([\kappa])$ is injective. 
Hence, $(R_{\calS})_x^\wedge$ is isomorphic to a power series over $k$ in $g$ variables modulo $r$ relations with $r \le \dim_k H^2(F(S)/F^+,\ad(r_x))$. 
Since $\rbar$ is Schur, $\ad(\rbar)^{G_{F^+}} = 0$ (see \cite{CHT}*{Lemma~2.1.7}), which implies $\ad(r_x)^{G_{F^+}} = 0$. 
The final claim now follows from \ref{thm:EPchar}.
\end{proof}

\begin{prop}\label{thm:smallglobalcomplete} 
Let the assumptions and notation be as in \ref{sec:globalcomplete}. 
Assume that $\mu$ is \mbox{de Rham}. 
Assume that for every $v|p$ in $F^+$, $v$ splits in $F$ and $\tildeS_p$ contains a place above $v$, which we will denote by  $\tilde{v}$. 
For each $\tilde{v} \in \tildeS_p$, we fix a Galois type $\tau_{\tilde{v}}$ and a $p$-adic Hodge type $\mathbf{v}_{\tilde{v}}$ over $E$ (see \ref{sec:types}). 
Take $R_{w} = R_{\tilde{v}}^\square(\tau_{\tilde{v}},\mathbf{v}_{\tilde{v}})$ (see \ref{thm:pdefring}) for $w = \tilde{v} \in \tildeS_p$, and $R_w = R_w^\square$ for each $w\in \tildeS \smallsetminus \tildeS_p$.
%
\begin{enumerate}
\item\label{smallglobalcomplete:rep} Let $D_{\calS,r_x}$ be the functor on $\Ar_k$ that sends an $\Ar_k$-algebra $B$ to the set of deformations $[r_B]$ of $r_x$ to $B$ that factor through $\Gal(F(S)/F^+)$, satisfy $\nu\circ r_B = \mu$, and are such that $r_B\rest_{G_{\tilde{v}}}$ is potentially semistable with Galois type $\tau_{\tilde{v}}$ and $p$-adic Hodge type $\mathbf{v}_{\tilde{v}}$ for each $\tilde{v}\in \widetilde{S}_p$. The functor $D_{\calS,r_x}$ is prorepresented by $(R_{\calS})_x^\wedge$.
\item\label{smallglobalcomplete:tangent}  The tangent space of $R_{\calS}[1/p]$ at $x$ is canonically isomorphic to $H_g^1(F(S)/F^+,\ad(r_x))$.
\end{enumerate}
\end{prop}

\begin{proof}
Since $R_w^\square$, for $w\in \tildeS\smallsetminus\widetilde{S}_p$, imposes no condition on our lifts, it is easy to see that if $\calS'$ is the deformation datum
 	\[\calS' = (F/F^+,S,\widetilde{S}_p,\calO,\rbar,\mu, 
 	\{R_{\tilde{v}}^\square(\tau_{\tilde{v}},\mathbf{v}_{\tilde{v}})\}_{\tilde{v} \in \widetilde{S}_p}), \]
there is a canonical isomorphism $R_{\calS} \cong R_{\calS'}$, and we may assume $\tildeS = \widetilde{S}_p$.  

Let $\calS^{\mathrm{big}} = (F/F^+,S,\tildeS_p,\calO,\rbar,\mu,\{R_{\tilde{v}}^\square\}_{\tilde{v}\in \tildeS_p})$, and let $R_{\calS^{\mathrm{big}}}$ be the universal type $\calS^{\mathrm{big}}$ deformation ring. 
Note $R_{\calS^{\mathrm{big}}}^{\loc} = R_{\tildeS_p}^\square$. 
By \ref{thm:globalrep}, choosing a lift $r_{\calS^{\mathrm{big}}}^{\univ}$ in the universal type $\calS^{\mathrm{big}}$ deformation yields a morphism $R_{\tildeS_p}^\square \rightarrow R_{\calS^{\mathrm{big}}}$. 
It is easy to see that $R_{\calS} \cong R_{\calS^{\mathrm{big}}} \otimes_{R_{\tildeS_p}^\square} R_{\calS}^{\loc}$, with $R_{\tildeS_p}^\square \rightarrow R_{\calS}^{\loc}$ the natural surjection. 
Using this and part~\ref{bigglobalcomplete:rep} of \ref{thm:bigglobalcomplete}, we see that $(R_{\calS})_x^\wedge$ prorepresents the functor on $\Ar_k$ that sends an $\Ar_k$-algebra $B$ to the set of deformations $[r_B]$ of $r_x$ to $B$ that factor through $\Gal(F(S)/F^+)$, satisfy $\nu\circ r_B = \mu$, and are such that the induced map 
	\[ R_{\tilde{v}}^\square \lra R_{\calS^{\mathrm{big}}} \lra (R_{\calS^{\mathrm{big}}})_x^\wedge \lra B \] 
factors through $R_{\tilde{v}}^\square(\tau_{\tilde{v}},\mathbf{v}_{\tilde{v}})$ for each $\tilde{v} \in \tildeS_p$, which happens if and only if $r_B\rest_{G_{\tilde{v}}}$ is potentially semistable with Galois type $\tau_{\tilde{v}}$ and $p$-adic Hodge type $\mathbf{v}_{\tilde{v}}$ by \ref{thm:pdefring}. 
This shows part~\ref{smallglobalcomplete:rep}.

We now show part~\ref{smallglobalcomplete:tangent}. 
Part~\ref{smallglobalcomplete:rep} together with part~\ref{bigglobalcomplete:tangent} of \ref{thm:bigglobalcomplete}, implies that the tangent space of $R_{\calS}[1/p]$ at $x$ is canonically isomorphic to the subspace of $H^1(F(S)/F^+,\ad(r_x))$ consisting of cohomology classes $[\kappa]$ such that for any cocycle $\kappa$ in the class $[\kappa]$, the local lift
%
	\[ r_\kappa\rest_{G_{\tilde{v}}} = ((1+\varepsilon \kappa)r_x)\rest_{G_{\tilde{v}}} = (1+\varepsilon \kappa|_{G_{\tilde{v}}})(r_x\rest_{G_{\tilde{v}}}) \]
lies in the tangent space of $R_{\tilde{v}}^\square(\tau_{\tilde{v}},\mathbf{v}_{\tilde{v}})[1/p]$ at $x$, for each $\tilde{v} \in \tildeS_p$.  
By \ref{thm:localBK}, this happens if and only if
$[\kappa]$ lies in the kernel of
	\[ H^1(F(S)/F^+,\ad(r_x)) \rightarrow \prod_{\tilde{v}\in \widetilde{S}_p} H^1(F_{\tilde{v}},B_{\dR}\otimes_{\Q_p}\ad(r_x)). \]
Since $\mu$ is \mbox{de Rham}, $r_\kappa\rest_{G_{\tilde{v}^c}} \cong (r_\kappa\rest_{G_{\tilde{v}}})^c \cong (r_\kappa\rest_{G_{\tilde{v}}})^\vee \otimes \mu$ is \mbox{de Rham} for each $\tilde{v}\in \widetilde{S}_p$. 
Thus, the tangent space of $R_{\calS}[1/p]$ at $x$ is canonically isomorphic to
	\begin{align*} 
	&\ker \big( H^1(F(S)/F^+,\ad(r_x)) \rightarrow \prod_{\tilde{v} \in \widetilde{S}_p} H^1(F_{\tilde{v}},B_{\dR} \otimes_{\Q_p}\ad(r_x)) \big) \\
	&\quad\quad\quad = \ker \big( H^1(F(S)/F^+,\ad(r_x)) \rightarrow \prod_{w|p \text{ in } F} H^1(F_w,B_{\dR}\otimes_{\Q_p}\ad(r_x)) \big) \\
	&\quad\quad\quad = \ker \big( H^1(F(S)/F^+,\ad(r_x)) \rightarrow \prod_{v|p \text{ in } F^+} H^1(F_v^+,B_{\dR} \otimes_{\Q_p}\ad(r_x)) \big). \qedhere
	\end{align*} 
\end{proof}

We note that in \ref{thm:smallglobalcomplete} above, the existence of $x \in \Spec R_{\calS}[1/p]$ implicitly assumes that $R_{\calS}[1/p] \ne 0$, which (at the very least) implies a certain compatibility between $\mu$ and the local $p$-adic Hodge theory data $\tau_{\tilde{v}}$ and $\mathbf{v}_{\tilde{v}}$. 
In our applications, we will be given a potentially semistable $p$-adic representation $\rho$ and we will define $\calS$ using $\rho$. 
The existence of this $\rho$ will then imply $R_{\calS}[1/p] \ne 0$. 

\begin{prop}\label{thm:globalsmooth} 
Let the notation and assumptions be as in \ref{sec:globalcomplete}, with $R_w = R_w^\square$ for every $w\in \tildeS$.
Assume also that $\mu(c) = -1$ for every choice of complex conjugation $c \in G_{F^+}$, and that 
$r_x\rest_{G_w} : G_w \lra \GL_d(k)$ 
is \mbox{de Rham} with regular $p$-adic Hodge type (see \ref{sec:types}) for every $w|p$ in $F$.

If $H_g^1(F(S)/F^+,\ad(r_x)) = 0$, then $(R_{\calS})_x^\wedge$ is formally smooth of dimension $\frac{d(d+1)}{2}[F^+:\Q]$.	
\end{prop}

\begin{proof}
%
Using our assumption that $\mu(c) = -1$ for every complex conjugation $c \in G_{F^+}$, part~\ref{bigglobalcomplete:dim} of \ref{thm:bigglobalcomplete} implies 
	\[ \dim (R_{\calS})_x^\wedge \ge  \frac{d(d+1)}{2}[F^+:\Q]. \]
Thus, it suffices to show the tangent space of $(R_{\calS})_x^\wedge$, which is isomorphic to $H^1(F(S)/F^+,\ad(r_x))$ by part~\ref{bigglobalcomplete:tangent} of \ref{thm:bigglobalcomplete}, has $k$-dimension $\frac{d(d+1)}{2}[F^+:\Q]$. 
This follows from part~\ref{cohom:H1dim} of \ref{thm:cohom}.  
\end{proof}

\include*{AuxPrimes}

\section{Automorphic theory}\label{sec:auttheory}

This section reviews the automorphic theory that we will use to prove our main theorems. 
We first introduce some notation and assumptions that will be used throughout this section. 


Let $\Z_+^d$ be the set of tuples of integers $(\lambda_1,\ldots,\lambda_d)$ such that $\lambda_1 \ge \cdots \ge \lambda_d$. 
If $F$ is a CM field with maximal totally real subfield $F^+$, and $\Omega$ is any characteristic $0$ field containing all embeddings of $F$ into an algebraic closure of $\Omega$, then we let $(\Z_+^d)_0^{\Hom(F,\Omega)}$ denote the subset of $(\Z_+^d)^{\Hom(F,\Omega)}$ of tuples $(\lambda_\tau)_{\tau\in \Hom(F,\Omega)}$ such that $\lambda_{\tau\circ c,i} = -\lambda_{\tau,d+1-i}$, where $c$ is the nontrivial element of $\Gal(F^+/F)$. 

If $F$ is a number field, $\chi : F^\times \backslash \A_F^\times \rightarrow \C^\times$ is a continuous character whose restriction to the connected component of $(F\otimes \R)^\times$ is given by $x \mapsto \prod_{\tau \in \Hom(F,\C)} x_\tau^{\lambda_\tau}$ for some integers $\lambda_\tau$, and $\iota : \Qbar_p \xrightarrow{\sim} \C$ is an isomorphism, we let $\chi_\iota : G_F \rightarrow \Qbar_p^\times$ be the continuous character given by 
	\[ \chi_\iota(\Art_F(x)) = \iota^{-1} \Big( \chi(x) \prod_{\tau \in \Hom(F,\C)} x_\tau^{-\lambda_\tau} \Big) 
	\prod_{\sigma\in \Hom(F,\Qbar_p)} x_\sigma^{\lambda_{\iota\sigma}}. \]

\subsection{Automorphic Galois representations}\label{sec:autgalrep}
Let $F$ be either a CM or totally real number field with maximal totally real subfield $F^+$. 
Let $c \in G_{F^+}$ be a choice of complex conjugation. 

\subsubsection{}\label{sec:pol} 
Following \cite{BLGGT}*{\S2.1}, we say that a pair $(\Pi,\chi)$ is a \emph{polarized automorphic representation} of $\GL_d(\A_F)$ if
\begin{itemize}
	\item $\Pi$ is an automorphic representation of $\GL_d(\A_F)$;
	\item $\chi : (F^+)^\times \backslash \A_{F^+}^\times \rightarrow \C^\times$ is a continuous character such that for all $v|\infty$, the value $\chi_v(-1)$ is independent of $v$, and equals $(-1)^d$ if $F$ is CM;
	\item $\Pi^c \cong \Pi^\vee \otimes (\chi\circ \Nm_{F/F^+}\circ \det)$.
	\end{itemize}
We say that an automorphic representation $\Pi$ of $\GL_d(\A_F)$ is \emph{polarizable} if there is a character $\chi$ such that $(\Pi,\chi)$ is a polarized automorphic representation.	
If $F$ is a totally real field and 
$(\Pi,1)$ is polarized, then we say that $\Pi$ is \emph{self-dual}.  
If $F$ is CM and 
$(\Pi,\delta_{F/F^+}^d)$ is polarized, then we say that $\Pi$ is \emph{conjugate self-dual}. 
Recall that an automorphic representation $\Pi$ of $\GL_d(\A_F)$ is called \emph{regular algebraic} if $\Pi_\infty$ has the same infinitesimal character as an irreducible algebraic representation of $\Res_{F/\Q} \GL_d$. 
If $\lambda = (\lambda_\tau) \in (\Z_+^d)^{\Hom(F,\C)}$, then we let $\xi_\lambda$ denote the irreducible algebraic representation of $\prod_{\tau}\GL_d$ which is the tensor product over $\tau \in \Hom(F,\C)$ of the irreducible algebraic representations with highest weight $\lambda_\tau$. 
We say a regular algebraic automorphic representation $\Pi$ of $\GL_d(\A_F)$ has \emph{weight} $\lambda \in (\Z_+^d)^{\Hom(F,\C)}$ if $\Pi_\infty$ has the same infinitesimal character as $\xi_\lambda^\vee$. 
We will say a polarized automorphic representation $(\Pi,\chi)$ of $\GL_d(\A_F)$ is cuspidal if $\Pi$ is. 
We will say a polarized automorphic representation $(\Pi,\chi)$ of $\GL_d(\A_F)$ is regular algebraic if $\Pi$ is. 
In this case $\chi$ is necessarily an algebraic character.


We have the following theorem, due to the work of many people. 
We refer the reader to \cite{BLGGT}*{Theorem~2.1.1} and the references contained there (noting that the assumption of an Iwahori fixed vector in part (4) of \cite{BLGGT}*{Theorem~2.1.1} can be removed by the main result of \cite{Caraianilp}). 


\begin{thm}\label{thm:autgalrep}
Let $F$ be either a CM or totally real number field with maximal totally real subfield $F^+$.
Let $(\Pi,\chi)$ be a regular algebraic, polarized, cuspidal automorphic representation of $\GL_d(\A_F)$, of weight $\lambda \in (\Z_+^d)^{\Hom(F,\C)}$. 
Fix a rational prime $p$ and an isomorphism $\iota : \Qbar_p \xrightarrow{\sim} \C$. 
Then there is a continuous semisimple representation 
	\[ \rho_{\Pi,\iota} : G_F \lra \GL_d(\Qbar_p) \]
satisfying the following properties. 
\begin{enumerate}
\item\label{autgalrep:pol} 
There is a perfect symmetric pairing $\langle \cdot, \cdot \rangle$ on $\Qbar_p^d$ 
such that for any $a,b\in \Qbar_p^d$ and $\sigma\in G_F$,
	\[ \langle \rho_{\Pi,\iota}(\sigma) a, \rho_{\Pi,\iota}(c\sigma c) b \rangle = (\epsilon^{1-d}\chi_\iota)(\sigma)\langle a, b \rangle. \] 
\item\label{autgalrep:pss} For all $w|p$, $\rho_{\Pi,\iota}|_{G_w}$ is potentially semistable, and for any continuous embedding $\tau : F_w \hookrightarrow \Qbar_p$, 
%
	\[ \HT_\tau(\rho_{\Pi,\iota}|_{G_w}) = \{ \lambda_{\iota\tau,j} + d - j\}_{j=1,\ldots,d}. \]
\item\label{autgalrep:locglob} For any finite place $w$, 
	\[\iota\WD(\rho_{\Pi,\iota}|_{G_w})^{\Fss} \cong \rec_{F_w}^T(\Pi_w). \]
\end{enumerate}
\end{thm}

We note that an argument using the Baire category theorem and the compactness of $G_F$ shows that we can assume $\rho_{\Pi,\iota}$ takes values in $\GL_d(\calO)$ with $\calO$ the ring of integers of some finite extension of $\Q_p$, and that the perfect pairing $\langle \cdot, \cdot \rangle$ descends to a perfect pairing on $\calO^d$.

\subsection{Definite unitary groups}\label{sec:unitary}
In this subsection, we assume that $p$ is odd. 
We recall some constructions from \cite{CHT}*{\S3.3} and \cite{GuerberoffModLift}*{\S2} (see also \cite{ThorneAdequate}*{\S6}). 
Before doing so, we note that in \cite{CHT}*{\S3.3} there are running assumptions that a certain ramification set denoted $S(B)$ there is nonempty, that (in our notation) $p > d$, and that $F$ is the composite of a quadratic imaginary field and $F^+$. 
None of these assumptions are necessary for what we need. 

Let $F$ be a CM field with maximal totally real subfield $F^+$. 
Let $c$ denote the nontrivial element of $\Gal(F/F^+)$. 
We assume that $F/F^+$ is unramified at all finite places and that every place above $p$ in $F^+$ splits in $F$. 
We assume that $F^+ \ne \Q$ and that if $d$ is even, then
	\[ d[F^+:\Q] \equiv 0 \pmod 4. \]
We also fix a finite extension $E$ of $\Q_p$ with ring of integers $\calO$ and residue field $\F$. 
We assume $E$ contains all embeddings of $F$ into an algebraic closure of $E$.

\subsubsection{}\label{sec:invol} 
Let $B = \M_d(F)$, and let $g \mapsto g^*$ be an involution of the second kind on $B$. 
Then the pair $(B,{}^*)$ defines a reductive $F^+$-group $G$ by
	\[ G(R) = \{ g\in B\otimes_{F^+} R \mid gg^* = 1\} \]
for any $F^+$-algebra $R$. 
Since $d[F^+:\Q] \equiv 0 \pmod 4$ if $d$ is even, we can choose the involution $g \mapsto g^*$ on $B$ such that 
\begin{enumerate}[label=(\alph*)]
	\item $G\otimes_{F^+} F_v^+$ is quasi-split for every finite place $v$ of $F^+$;
	\item $G(F_v^+) \cong U_d(\R)$, the totally definite unitary group, for every infinite place $v$ of $F^+$.
\end{enumerate}
Since $B = \M_d(F)$, the data of ${}^*$ is equivalent to a Hermitian form $h$ on $F^d$, and a choice of lattice $\calL$ in $F^d$ such that $h(\calL \times \calL) \subseteq \calO_F$ yields a maximal order 
$\calO_B$ of $B$ such that $\calO_B^* = \calO_B$. 
This maximal order defines a model of $G$ over $\calO_{F^+}$, that we again denote by $G$. 
For a finite place $v$ of $F^+$ that splits in $F$, 
we can find an isomorphism $\iota_v : \calO_{B,v} \xrightarrow{\sim} \M_d(\calO_{F}) \otimes_{\calO_{F^+}}\calO_{F_v^+}$ such that $\iota_v(g^*) = {}^t\iota_v(g)^c$. 
Writing $v = ww^c$, the choice of $w$ gives an isomorphism $\iota_w : G(\calO_{F_v^+}) \xrightarrow{\sim} \GL_d(\calO_{F_w})$ by $\iota_v^{-1}(g,{}^t(g^c)^{-1}) \mapsto g$. 
This extends to an isomorphism $\iota_w : G(F_v^+) \xrightarrow{\sim} \GL_d(F_w)$, and $\iota_{w^c} = {}^t(c\circ \iota_w)^{-1}$. 

Let $S_p$ denote the set of places of $F^+$ above $p$. 
For each $v\in S_p$, fix a choice $\widetilde{v}$ of place of $F$ dividing $v$, and let $\widetilde{S}_p = \{\widetilde{v} \mid v\in S_p\}$. 
Let $J_p$ be the set of embeddings $F^+ \hookrightarrow E$, and let $\widetilde{J}_p$ be the set of embeddings $F \hookrightarrow E$ that give rise to a place in $\widetilde{S}_p$. 
Thus, restriction to $F^+$ gives a bijection $\widetilde{J}_p \xrightarrow{\sim} J_p$.
Given $\lambda \in \Z_+^d$, we let
	\[ \xi_\lambda : \GL_d \lra \GL(W_\lambda) \]
denote the irreducible algebraic representation defined over $\Q$ with highest weight 
	\[ \mathrm{diag}(t_1,\ldots,t_d) \lra \prod_{i = 1}^d t_i^{\lambda_i}. \]
We choose a $\GL_d(\calO)$-stable lattice $M_\lambda$ in $W_\lambda \otimes_\Q E$. 
Then for any $\lambda = (\lambda_\tau) \in (\Z_+^d)_0^{\Hom(F,E)}$, we define a representation
	\begin{align*} \xi_\lambda : G(F_p^+) &\lra \GL(\otimes_{\tau\in \widetilde{J}_p} W_{\lambda_\tau}) \\ 
	g & \longmapsto \bigotimes_{\tau\in \widetilde{J}_p} \xi_{\lambda_\tau}(\tau \iota_\tau g),
	\end{align*}
where we have written $\iota_\tau$ for $\iota_w$ if $\tau$ gives rise to the place $w$. 
The $\calO$-lattice $M_\lambda : = \otimes_{\tau\in \widetilde{J}_p} M_{\lambda_\tau}$ in the $E$-vector space $W_\lambda : = \otimes_{\tau\in \widetilde{J}_p} W_{\lambda_\tau}$ is stable under $G(\calO_{F_p^+})$. 

If $A$ is any $\calO$-algebra and $U$ is an open compact subgroup of $G(\A_{F^+}^\infty)$ such that $U_p \subseteq G(\calO_{F_p^+})$, then we define a space of automorphic forms $S_\lambda(U,A)$ to be the space of functions
	\[ f : G(F^+) \backslash G(\A_{F^+}^\infty) \lra M_\lambda \otimes_{\calO} A \]
such that $f(gu) = u_p^{-1} f(g)$ for all $g\in G(\A_{F^+}^\infty)$ and $u \in U$. 
If $V$ is any compact subgroup of $G(\A_{F^+}^\infty)$ such that $V_p \subseteq G(\calO_{F_p^+})$, then we define
	\[ S_\lambda(V,A) = \varinjlim_{V\subseteq U} S_\lambda(U,A),\]
with the limit over all open compact subgroups $U$ containing $V$ such that $U_p \subseteq G(\calO_{F_p^+})$. 
Note that if $A$ is flat over $\calO$, then $S_\lambda(V,A) = S_\lambda(V,\calO) \otimes_{\calO} A$. 

\begin{prop}\label{thm:unitaryaut}
Fix an embedding $\iota : E \hookrightarrow \C$, and view $\C$ as an $\calO$-algebra via $\iota$. 
\begin{enumerate}
\item\label{unitaryaut:adm} $S_\lambda(\{1\},\C)$ is a semisimple admissible representation of $G(\A_{F^+}^\infty)$.
\item\label{unitaryaut:BC} Let $\Pi$ be an regular algebraic conjugate self-dual representation of $\GL_d(\A_F)$ of weight $\iota\lambda$. 
Then there is an irreducible subrepresentation $\pi = \otimes_v\pi$ of $S_\lambda(\{1\},\C)$ such that the following hold. 
	\begin{itemize}
	\item If $v$ is a finite place of $F^+$ that splits as $v = ww^c$ in $F$, then $\pi_v \cong \Pi_w \circ\iota_w$.
	\item If $v$ is a finite place of $F^+$ inert in $F$, and $\Pi_v$ is unramified, then $\pi_v$ has a fixed vector for some  hyperspecial maximal compact subgroup of $G(F_v^+)$.
	\end{itemize}
\end{enumerate}
\end{prop}

\begin{proof} 
For part~\ref{unitaryaut:adm}, see \cite{CHT}*{part~1 of Proposition~3.3.2}. 
For part~\ref{unitaryaut:BC}, \cite{LabesseBCCM}*{Th\'{e}or\`{e}me~5.4} implies that there is an irreducible subrepresentation $\pi = \otimes_v\pi$ of $S_\lambda(\{1\},\C)$ such that $\Pi$ is a weak base change (see \cite{LabesseBCCM}*{\S4.10}) of $\pi$. 
Then applying \cite{LabesseBCCM}*{Corollaire~5.3} and strong multiplicity one for $\GL_d$ gives the result.
\end{proof}

\subsubsection{}\label{sec:Heckeops} 
Now fix an open compact subgroup $U$ of $G(\A_{F^+}^\infty)$ such that $U_p \subseteq G(\calO_{F_p^+})$. 
Fix a set of finite places $S$ of $F^+$ containing $S_p$ and all finite places at which $U$ is not hyperspecial maximal compact. 
For each finite place $w$ of $F$ split over $v \notin S$ of $F^+$, the Hecke operators
	\[ T_w^{(j)} = \iota_w^{-1} \Big(\big[ \GL_d(\calO_{F_w})
		\begin{pmatrix} \varpi_w 1_{j} & \\ & 1_{d-j} \end{pmatrix}
	\GL_d(\calO_{F_w}) \big] \Big) \]
for $j \in \{1,\ldots,d\}$, as well as $(T_w^{(d)})^{-1}$, act on $S_\lambda(U,A)$. 
Here $\varpi_w$ denotes a choice of uniformizer of $F_w$, and the operators $T_w^{(j)}$ do not depend in the choice of $\varpi_w$. 
We let $\bbT_\lambda^S(U,A)$ be the $A$-subalgebra of $\End_A(S_\lambda(U,A))$ generated by the $T_w^{(1)},\ldots,T_w^{(d)},(T_w^{(d)})^{-1}$, for all $w$ as above. 
If $A = \calO$, then we write $\bbT_\lambda^S(U)$ for $\bbT_\lambda^S(U,\calO)$. 

\begin{cor}\label{thm:galrepunit}
Let $L/E$ be finite extension, and $x : \bbT_\lambda^S(U)[1/p] \rightarrow L$ be an $E$-algebra morphism. 
There is a continuous absolutely semisimple representation
	\[ \rho_x : \Gal(F(S)/F) \lra \GL_d(L) \]
such that the following hold.
\begin{enumerate}
\item\label{galrepunit:pol}  
There is a perfect symmetric paring $\langle \cdot, \cdot \rangle$ on $L^d$ such that for any $a,b\in L^d$ and $\sigma\in G_F$,
	\[ \langle \rho_x(\sigma) a, \rho_x(c\sigma c) b \rangle = \epsilon^{1-d}(\sigma)\langle a, b \rangle. \]
\item\label{galrepunit:trace} For any finite place $w$ in $F$ split over $v \notin S$ in $F^+$, the characteristic polynomial of $\rho_x(\Frob_w)$ is 
	\[ X^d + \cdots  (-1)^j\Nm(w)^{\frac{j(j-1)}{2}} x(T_w^{(j)}) X^{d-j} + \cdots + (-1)^d\Nm(w)^{\frac{d(d-1)}{2}} 
	x(T_w^{(d)}). \]
\item\label{galrepunit:pss} For all $w|p$ in $F$, $\rho_x|_{G_w}$ is potentially semistable, and is semistable if $\iota_w(U_v)$ contains the Iwahori subgroup, where $v$ denotes the place of $F^+$ below $w$. 
For any $\tau \in \widetilde{J}_p$, 
	\[ \HT_\tau(\rho_x|_{G_w}) = \{ \lambda_{\tau,j} + d - j\}_{j = 1,\ldots,d}. \]
In particular, $\rho_x|_{G_w}$ has regular $p$-adic Hodge type.
\end{enumerate}
%
\end{cor}

\begin{proof}
The finite flat $\calO$-algebra $\bbT_\lambda^S(U)$ acts faithfully on 
	\[ S_\lambda(U)\otimes_{\calO} \Qbar_p \cong S_\lambda(U,\Qbar_p) = \bigoplus_{\pi} \pi^U, \]
with the direct sum taken over irreducible subrepresentations of $S_\lambda(\{1\},\Qbar_p)$, 
and $\bbT_\lambda(U)$ acts on each $\pi^U$ by some $E$-algebra morphism $x_\pi : \bbT_\lambda(U)[1/p] \rightarrow \Qbar_p$. 
Thus, $x \otimes_L \Qbar_p = x_\pi$ for some irreducible subrepresentation $\pi$ of $S_\lambda(\{1\},\Qbar_p)$. 
We now apply \cite{GuerberoffModLift}*{Theorem~2.3}, letting $\rho_x$ be the representation denoted $r_p(\pi)$ in \cite{GuerberoffModLift}*{Theorem~2.3}, noting that:
\begin{itemize}
	\item We may assume that $\rho_x$ takes values in $L$ by Chebotarev density, since $x_\pi(T_w^{(j)}) = x(T_w^{(j)}) \in L$ for all $w$ of $F$ split over $v\notin S$ in $F^+$ and $j \in \{1,\ldots,d\}$.
	\item We may assume that the pairing is symmetric by applying \cite{BellaicheChenevierSign}*{Theorem~1.2} to the construction in the proof of \cite{GuerberoffModLift}*{Theorem~2.3}. More specifically, the construction in \cite{GuerberoffModLift}*{Theorem~2.3} shows there is a partition $d = d_1+\cdots+d_r$ and factorizations $d_i = a_ib_i$, such that
		\[ \rho_x = \bigoplus_{i=1}^r \varrho_i \otimes \eta_i \otimes(1\oplus \epsilon \oplus \cdots \oplus \epsilon^{b_i-1}),\]
	where $\varrho_i$ is the Galois representation associated to a regular algebraic conjugate self-dual cuspidal automorphic representation of $\GL_{a_i}(\A_F)$, and $\eta$ is a character satisfying $\eta\eta^c = \epsilon^{1 + a_i - b_i - d}$. 
	There is a symmetric perfect pairing for $\varrho_i$ with multiplier $\epsilon^{1-a_i}$ by part~\ref{autgalrep:pol} of \ref{thm:autgalrep} (which uses \cite{BellaicheChenevierSign}*{Theorem~1.2}). 
	This defines a symmetric perfect pairing for $\varrho_i\otimes\eta$ with multiplier $\epsilon^{2 - b_i - d}$. 
	Defining the obvious symmetric pairing on $(1\oplus \cdots \oplus \epsilon^{b_i-1})$ with multiplier $\epsilon^{b_i-1}$, tensoring these two, and taking the direct sum over $i$ yields a symmetric perfect pairing for $\rho_x$ with multiplier $\epsilon^{1-d}$.\qedhere
\end{itemize}
\end{proof}

%

Recall the group scheme $\calG_d$ and its canonical character $\nu : \calG_d \rightarrow \GL_1$ of \S\ref{sec:global}. 
The following is a standard application of \ref{thm:galrepunit} and \ref{thm:Gdhoms}. 

\begin{lem}\label{thm:residrep}
Let $\frakm$ be a maximal ideal of $\bbT_\lambda^S(U)$. 
There is a continuous homomorphism 
	\[ \rbar_{\frakm} : \Gal(F(S)/F^+) \lra \calG_d(\bbT_\lambda^S(U)/\frakm) \]
with $\rbar_{\frakm}^{-1}(\calG_d^0(\bbT_\lambda^S(U)/\frakm)) = \Gal(F(S)/F)$, satisfying the following:
\begin{enumerate}
	\item If $w$ is a finite place of $F$ split over $v \notin S$ of $F^+$, then $\rbar_{\frakm}(\Frob_w)$ has characteristic polynomial
		\[ X^d + \cdots (-1)^j\Nm(w)^{\frac{j(j-1)}{2}} T_w^{(j)} X^{d-j} + \cdots + (-1)^d\Nm(w)^{\frac{d(d-1)}{2}} T_w^{(d)} 
		\bmod \frakm. \]
	\item $\nu\circ \rbar_{\frakm} = \epsilon^{1-d}\delta_{F/F^+}^d \bmod {\frakm_\calO}$.
\end{enumerate}
\end{lem}
%

We now fix a maximal ideal $\frakm$ of $\bbT_\lambda^S(U)$, and let 
	\[ \rbar_\frakm : \Gal(F(S)/F^+) \lra \calG_d(\bbT_\lambda^S(U)/\frakm) \]
be as in \ref{thm:residrep}. 
Enlarging $E$ if necessary, we may assume that $\bbT_\lambda^S(U)/\frakm = \F$. 
Arguing exactly as in the proof of \cite{CHT}*{Proposition~3.4.4}, we have the following lemma. 

\begin{lem}\label{thm:galrephecke}
Assume that $\rbar_{\frakm}\rest_{G_F}$ is absolutely irreducible. 
Then there is a continuous lift 
	\[ r_{\frakm} : \Gal(F(S)/F^+) \lra \calG_d(\bbT_\lambda^S(U)_{\frakm}) \]
of $\rbar_\frakm$, unique up to conjugacy, satisfying the following:
\begin{enumerate}
	\item\label{galrephecke:frob} If $w$ is a finite place of $F$ split over $v \notin S$ of $F^+$, then $\rbar_{\frakm}(\Frob_w)$ has characteristic polynomial
		\[ X^d + \cdots (-1)^j\Nm(w)^{\frac{j(j-1)}{2}} T_w^{(j)} X^{d-j} + \cdots + (-1)^d\Nm(w)^{\frac{d(d-1)}{2}} T_w^{(d)}. \]
	\item $\nu\circ \rbar_{\frakm} = \epsilon^{1-d}\delta_{F/F^+}^d$.
\end{enumerate}
\end{lem}


For the remainder of this subsection we assume that $\rbar_{\frakm}\rest_{G_F}$ is absolutely irreducible. 
Fix an $E$-algebra morphism $x : \bbT_\lambda^S(U)_{\frakm}[1/p] \rightarrow E$. 
The representation $\rho_x$ of \ref{thm:galrepunit} is \mbox{de Rham}, and for each $w|p$, we let $\mathbf{v}_w : = D_{\dR}(\rho_x|_{G_w})$ be its $p$-adic Hodge type. 

Let $\tildeS$ be a finite set of places in $F$ each of which is split over some $v \in S$, containing $\widetilde{S}_p$, and such that $\tildeS \cap \tildeS^c = \emptyset$. 
We consider the global deformation datum
	\[\calS = (F/F^+,S,\tildeS,\calO,\rbar,\epsilon^{1-d}\delta_{F/F^+}^d, 
	\{R_{\tilde{v}}^\square(1,\mathbf{v}_{\tilde{v}})\}_{\tilde{v} \in \widetilde{S}_p} \cup 
	\{R_w^\square\}_{w \in \tildeS\smallsetminus \widetilde{S}_p}), \]
and let $R_{\calS}$ be the universal type $\calS$ deformation ring (see~\ref{thm:globalrep}). 

\begin{lem}\label{thm:heckedef}
Assume that $\iota_{\tilde{v}}(U_v)$ contains the Iwahori subgroup for each $v \in S_p$. 
\begin{enumerate}
\item\label{heckedef:typeS} The lift 
	\[ r_{\frakm} : \Gal(F(S)/F^+) \lra \calG_d(\bbT_\lambda^S(U)_{\frakm}) \]
of $\rbar$ from \ref{thm:galrephecke} is of type $\calS$, and the induced $\CNL_{\calO}$-morphism $R_{\calS} \rightarrow \T_\lambda^S(U)_{\frakm}$ is surjective.
\item\label{heckedef:free} Denote again by $x$ the composite $R_{\calS}[1/p] \rightarrow \bbT_\lambda^S(U)_{\frakm}[1/p] \xrightarrow{x} E$. 
The localization and completion $(R_{\calS})_x^\wedge$ acts faithfully on $S_\lambda(U)_x^\wedge$ if and only if $(R_{\calS})_x^\wedge = E$.
\end{enumerate}
\end{lem}

\begin{proof}
The fact that $r_{\frakm}$ factors through $\Gal(F(S)/F^+)$ and has $\nu\circ r_{\frakm} = \epsilon^{1-d}\delta_{F/F^+}^d$ is contained in \ref{thm:galrephecke}. 
To show that $r_{\frakm}$ is of type $\calS$, it remains to show that for every $w\in \tildeS$, the local lift $r_{\frakm}\rest_{G_w}$ factors through the given local lifting ring $R_w$. 
For $w \in \tildeS \smallsetminus \widetilde{S}_p$, this is automatic, so we only need to show it for $\tilde{v} \in \widetilde{S}_p$. 
Since $\bbT_\lambda^S(U)_{\frakm}$ is reduced and finite flat over $\calO$, it suffices to show that for any finite extension $L/E$, and any $E$-algebra morphism $y : \bbT_\lambda^S(U)_{\frakm}[1/p] \rightarrow L$, that the representation
	\[ \rho_y : \Gal(F(S)/F) \xrightarrow{r_{\frakm}\rest_{G_F}} \GL_d(\T_\lambda^S(U)_{\frakm}) \xrightarrow{y} \GL_d(L) \]
induces an $E$-algebra morphism $R_{\tilde{v}}^\square[1/p] \rightarrow L$ that factors through $R_{\tilde{v}}^\square(1, \mathbf{v}_{\tilde{v}})[1/p]$ for every $\tilde{v} \in \widetilde{S}_p$.  
This happens if and only if for every $\tilde{v} \in \widetilde{S}_p$, the local representation $\rho_y|_{G_{\tilde{v}}}$ is semistable with $p$-adic Hodge type $\mathbf{v}_{\tilde{v}}$. 
This follows from part~\ref{galrepunit:pss} of \ref{thm:galrepunit}, since $\iota_{\tilde{v}}(U_v)$ contains the Iwahori subgroup for each $v\in S_p$, and the $p$-adic Hodge type of $\rho_y|_{G_{\tilde{v}}}$ depends only on its associated graded, which in turns depends only on $\lambda$. 
This shows that $r_{\frakm}$ is of type $\calS$. 
That the induced $\CNL_{\calO}$-algebra map is surjective follows from part~\ref{galrephecke:frob} of \ref{thm:galrephecke} and Chebotarev density.
For part \ref{heckedef:free}, first note that since $S$ contains all places at which $U$ is not hyperspecial maximal compact, the finite dimensional $E$-vector space $S_\lambda(U)[1/p]$ is a faithful and semisimple $\T_\lambda^S(U)[1/p]$-module.
Then $S_\lambda(U)_x$ is a finite dimensional $E$-vector space spanned by the eigenforms with $\T_\lambda^S(U)[1/p]$-eigensytem $x$, and $\T_\lambda^S(U)_x \cong E$. It follows that $(R_{\calS})_x^\wedge$ acts on the $E$-vector space $S_\lambda(U)_x^\wedge = S_\lambda(U)_x$ through the surjection to its residue field $E$.
\end{proof}

\section{The main theorems}\label{sec:mainsec}

Throughout this section $p$ will be an odd prime and $E$ will be a finite extension of $\Q_p$ with ring of integers $\calO$ and residue field $\F$. 
Let $F$ be a CM field with maximal totally real subfield $F^+$. 
Let $c \in G_{F^+}$ be a fixed choice of complex conjugation. 
Let $S$ be a finite set of places of $F^+$ containing all places above $p$. 


\subsection{The main theorem}\label{sec:mainthm} 
In this subsection we prove our main theorem on adjoint Selmer groups. 
Before doing so, we first recall the definition of an adequate subgroup of $\GL_d(\F)$.

Let $\Gamma$ be a subgroup of $\GL_d(\F)$, and assume that $\F$ contains all eigenvalues of all elements of $\Gamma$. 
Let $\frakgl_d = \frakgl_d(\F)$ be the Lie algebra of $\GL_d(\F)$ with the adjoint $\Gamma$-action. 
Let $\frakz \subset \frakgl_d$ be the centre of $\frakgl_d$; 
note $\frakz \cong \F$.  
The following is \cite{Thorne2adic}*{Definition~2.20} (using \cite{ThorneAdequate}*{Lemma~A.1}).

\begin{defn}\label{def:GLadequate}
We say $\Gamma$ is \emph{adequate} if the following hold:
\begin{enumerate}
\item $H^1(\Gamma,\F) = 0$ and $H^1(\Gamma,\frakgl_d/\frakz) = 0$;
\item $\End_{\F}(\F^d) = \mathrm{M}_d(\F)$ is spanned by the semisimple elements in $\Gamma$.
\end{enumerate}
\end{defn}

We note that this is slightly more general than the definition of adequate found in much of the literature (e.g. \cite{ThorneAdequate}*{Definition~2.3}). 
In particular, it allows $p|d$. 
%
%
%
%

\subsubsection{}\label{sec:CMcase} 
Let
	\[ r : \Gal(F(S)/F^+) \lra \calG_d(E) \]
be a continuous homomorphism. 
We recall that $\ad(r)$ denotes $\mathfrak{gl}_d(E)$ with the adjoint $\Gal(F(S)/F^+)$-action $\ad\circ r$. 
After conjugating, we can assume that $r$ takes values in $\calG_d(\calO)$, and we let $\rbar = r \otimes_{\calO} \F$. 
%
%

\begin{thm}\label{thm:BK}
Assume there is a finite extension $L/F$ of CM fields, a regular algebraic polarizable cuspidal automorphic representation $\Pi$ of $\GL_d(\A_{L})$, and an isomorphism $\iota : \Qbar_p \xrightarrow{\sim} \C$ such that the following hold:
	\begin{enumerate}[label=(\alph*)]
	\item\label{BK:aut} $r\rest_{G_{L}} \!\otimes\, \Qbar_p \cong \rho_{\Pi,\iota}$;
	\item\label{BK:adequate} $\zeta_p \notin L$ and $\rbar|_{G_{L(\zeta_p)}}$ has adequate image.
	\end{enumerate}
Then the following hold.
\begin{enumerate}
\item\label{BK:BK} $H_g^1(F(S)/F^+,\ad(r)) = 0$.
\item\label{BK:H2} $H^2(F(S)/F^+,\ad(r)) = 0$.
\item\label{BK:H1} The natural map
	\[ H^1(F(S)/F^+,\ad(r)) \lra \prod_{v|p} H^1(F_v^+,\ad(r)) / H_g^1(F_v^+,\ad(r)) \]
is an isomorphism.
\end{enumerate}
\end{thm}

\begin{proof}
Using \ref{thm:cohom}, parts~\ref{BK:H2} and~\ref{BK:H1} are implied by \ref{BK:BK}. 
To prove $H_g^1(F(S)/F^+, \ad(r)) = 0$, we first note that we are free to enlarge $S$, so we can (and do) assume that $S$ contains at least one finite place not above $p$. 
We are also free to replace $F$ with any finite extension $L/F$ of CM fields
Indeed, let $L^+$ denote the maximal totally real subfield of $L$ and $S_{L^+}$ the set of places of $L^+$ above $S$. 
Then $F(S)L^+$ is Galois over $L^+$ and $\Gal(F(S)L^+/L^+)$ canonically isomorphic to an open subgroup $H$ of $\Gal(F(S)/F^+)$. 
The restriction on cohomology to $H$ is injective by considering restriction-corestriction, then via the isomorphism $H \cong \Gal(F(S)L^+/L^+)$ and inflation to $\Gal(L(S_{L^+})/L^+)$, we have an injection 
	\[ H^1(F(S)/F^+,\ad(r)) \lra H^1(L(S_{L^+})/L^+,\ad(r)) \]
such that the following diagram commutes
	\[ \begin{tikzcd}
	H^1(F(S)/F^+,\ad(\rho)) \arrow{r} \arrow{d} & \prod_{v|p} H^1(F_v^+,B_{\dR} \otimes_{\Q_p} \ad(r)) \arrow{d} \\
	H^1(L(S_{L^+})/L^+,\ad(r)) \arrow{r} & \prod_{w|p} H^1(L_w^+,B_{\dR} \otimes_{\Q_p} \ad(r)).
	\end{tikzcd}
	\]
Thus $H_g^1(F(S)/F^+,\ad(r))$ injects into $H_g^1(L(S_{L^+})/L^+,\ad(r))$, so the former is trivial if the latter is. 
In particular, we may assume $L = F$ in the statement of the theorem. 

Note that if $\chi : \Gal(F(S)/F) \rightarrow \calO^\times$ is any continuous character, then by \ref{thm:Gdhoms} there is a continuous homomorphism $r \otimes \chi : \Gal(F(S)/F^+) \rightarrow \calG_d(\calO)$ such that $(r \otimes \chi) \rest_{G_F} = r\rest_{G_F} \otimes\, \chi$ and $\nu \circ (r\otimes \chi) = (\chi\chi^c) (\nu\circ r)$. 
There is an isomorphism $\ad(r) \cong \ad(r\otimes \chi)$ of $\Gal(F(S)/F^+)$-modules, so using \cite{CHT}*{Lemma~4.1.4}, we can twist $r$ and $\Pi$, and assume $\Pi$ is conjugate self-dual. 

Let $S_p$ be the set of places above $p$ in $F^+$, and for each $v \in S_p$, choose some $\tilde{v}|v$ in $F$. 
Let $\widetilde{S}_p = \{\tilde{v} \mid v\in S_p \}$. 
Using cyclic base change \cite{ArthurClozel}*{Theorem~4.2 of Chapter~3}, further replacing $F$ by a solvable extension, we may assume that $F/F^+$, $r$, $\rbar$, and $\Pi$ satisfy:
	\begin{enumerate}[label=(\roman*)]
	\item $F/F^+$ is unramified at all finite places, and every $v\in S$ splits in $F$;
	\item if $d$ is even, then $d[F^+:\Q] \equiv 0 \pmod 4$;
	\item\label{lotsofass:galrepaut} $r\rest_{G_F} \otimes \,\Qbar_p\cong \rho_{\Pi,\iota}$, and $\Pi$ is conjugate self-dual;
	\item $\zeta_p \notin F$ and $\rbar(G_{F(\zeta_p)})$ is adequate;
	\item $\Pi_{\tilde{v}}$ has Iwahori fixed vectors for every $\tilde{v} \in \widetilde{S}_p$.
	\end{enumerate}
Enlarging $E$ if necessary, we assume that $E$ contains all embedding of $F$ into $\Qbar_p$, and that $\F$ contains the eigenvalues of all elements in the image of $\rbar\rest_{G_F}$. 
Let $\lambda \in (\Z_+^d)^{\Hom(F,E)}$ be such that $\iota\lambda$ is the weight of $\Pi$. 
Let $G$ be the $\calO_{F^+}$-group scheme of \ref{sec:invol}. 
Recall that for any finite place $w$ of $F$ split over $F^+$, we have an isomorphism $\iota_w : G(\calO_{F_v^+}) \xrightarrow{\sim} \GL_d(\calO_{F_w})$. 
For any $\calO$-algebra $A$, and compact subgroup $V$ of $G(\A_{F^+}^\infty)$ with $V_p \subseteq G(\calO_{F_p^+})$, we let $S_\lambda(V,A)$ be the $A$-module of automorphic forms of \ref{sec:invol}. 
Viewing $\C$ as an $\calO$-algebra via $\iota$, by \ref{thm:unitaryaut} there is an irreducible subrepresentation $\pi$ of $S_\lambda(\{1\},\C)$ such that
	\begin{enumerate}[resume,label=(\roman*)]
	\item if $v$ is a finite place of $F^+$ that splits as $v = ww^c$ in $F$, then $\pi_v \cong \Pi_w \circ\iota_w$;
	\item\label{lotsofass:basechange} if $v$ is a finite place of $F^+$ inert in $F$, and $\Pi_v$ is unramified, then $\pi_v$ has a fixed vector for some  hyperspecial maximal compact subgroup of $G(F_v^+)$.
	\end{enumerate}
Now choose an open compact subgroup $U \subseteq G(\A_{F^+}^\infty)$ such that $U_p \subseteq G(\calO_{F_p^+})$ and satisfying the following:
	\begin{enumerate}[resume,label=(\roman*)]
	\item\label{lotsofass:Iw} $\iota_{\tilde{v}}(U_v)$ is the Iwahori subgroup of $\GL_d(\calO_{F_{\tilde{v}}})$ for each $v\in S_p$;
	\item\label{lotsofass:Ufixed} $\pi^U \ne 0$;
	\item\label{lotsofass:hyper} $U_v$ is a hyperspecial maximal compact subgroup for every $v\notin S$;
	\item\label{lotsofass:small} for all $t\in G(\A_{F^+}^\infty)$, the group $t^{-1}G(F^+)t \cap U$ contains no element of order $p$.
	\end{enumerate}
Since we have assumed that $S$ contains at least one finite place not above $p$, assumptions \ref{lotsofass:Iw}, \ref{lotsofass:hyper}, and \ref{lotsofass:small} can be satisfied simultaneously by letting $U_u$ be sufficiently small for some $u \in S \smallsetminus S_p$. 
Let $\bbT_\lambda^S(U)$ be the Hecke algebra in \ref{sec:Heckeops}. 
Since $S_\lambda(U,\C) \cong S_\lambda(U)\otimes_{\calO,\iota} \C$, there is an action of $\bbT_\lambda^S(U)$ on $\pi^U$ and it acts via an $E$-algebra morphism $x : \bbT_\lambda^S(U) \rightarrow E$. 
Indeed, $\bbT_\lambda^S(U)$ acts on $\pi^U$ via a homomorphism $x : \bbT_\lambda^S(U)[1/p] \rightarrow \C$ that factors through $E$ by assumptions \ref{lotsofass:galrepaut} and \ref{lotsofass:basechange}, local-global compatibility (i.e. part~\ref{autgalrep:locglob} of~\ref{thm:autgalrep}), and that $r\rest_{G_F}$ takes values in $\GL_d(E)$. 
Moreover, the Galois representation $\rho_x : \Gal(F(S)/F^+) \rightarrow \GL_d(E)$ of \ref{thm:galrepunit} is isomorphic to $ r\rest_{G_F}$. 

Let $\frakm$ be the maximal ideal of $\bbT_\lambda^S(U)$ containing $\ker(x)$. 
We can choose a continuous homomorphism 
	\[ \rbar_{\frakm} : \Gal(F(S)/F^+) \lra \calG_d(\bbT_\lambda^S(U)/\frakm) = \calG_d(\F) \]
as in \ref{thm:residrep}, such that $\rbar_{\frakm} = \rbar$. 
In particular, $\rbar_{\frakm}\rest_{G_F}$ is absolutely irreducible. 
Then \ref{thm:galrephecke} gives a lift 
	\[ r_{\frakm} : \Gal(F(S)/F^+) \lra \calG_d(\bbT_\lambda^S(U)_{\frakm}) \]
of $\rbar_{\frakm}$, and letting $r_x$ be the composite
	\[ \Gal(F(S)/F^+) \xrightarrow{r_{\frakm}} \calG_d(\bbT_\lambda^S(U)_{\frakm}) \stackrel{x}{\lra}
	\calG_d(E),\]
the lifts $r_x$ and $r$ of $\rbar$ define the same deformation.

For each $v \in S\smallsetminus S_p$, let $\tilde{v}$ be a choice of place in $F$ dividing $v$, and let $\widetilde{S} = \{\tilde{v} \mid v \in S \}$. 
For each $\tilde{v} \in \widetilde{S}_p$, let $\mathbf{v}_{\tilde{v}}  = D_{\dR}(\rho|_{G_{\tilde{v}}})$ be the $p$-adic Hodge type of $\rho$ at $\tilde{v}$. 
We consider the global deformation datum (see~\ref{def:globaldatum})
	\[\calS = (F/F^+,S,\tilde{S},\calO,\rbar,\epsilon^{1-d}\delta_{F/F^+}^d, \{R_{\tilde{v}}\}_{\tilde{v}\in \tildeS}) \]
where
	\begin{itemize}
	\item $R_{\tilde{v}} = R_{\tilde{v}}^\square(1,\mathbf{v}_{\tilde{v}})$ if $\tilde{v} \in \tildeS_p$ (see~\ref{thm:pdefring}) and
	\item $R_{\tilde{v}} = R_{\tilde{v}}^\square$ if $\tilde{v} \in \tilde{S}\smallsetminus \widetilde{S}_p$.
	\end{itemize}
We let $R_{\calS}$ 
be the universal type $\calS$ deformation
ring (see~\ref{def:globaldef} and~\ref{thm:globalrep}), and let $R_{\calS}^{\loc} = \widehat{\otimes}_{\tilde{v} \in \tildeS} R_{\tilde{v}}$. 
By part~\ref{heckedef:typeS} of \ref{thm:heckedef}, the deformation $[r_{\frakm}]$ of $\rbar_{\frakm}$ is of type $\calS$, and there is a surjective $\CNL_{\calO}$-algebra morphism $R_{\calS} \rightarrow \bbT_\lambda(U)_{\frakm}$. 
Again denote by $x$ the $E$-algebra morphism $R_{\calS}[1/p] \rightarrow \bbT_\lambda(U)_{\frakm}[1/p] \xrightarrow{x} E$. 
Make a choice of lift for the universal type $\calS$ deformation so that the specialization $r_x$ of this lift via $x : R_{\calS}[1/p] \rightarrow E$ is equal to $r$. 
Then, by \ref{thm:smallglobalcomplete}, 
	\[ H_g^1(F(S)/F^+,\ad(r)) 
	= 0 \]
if and only if $(R_{\calS})_x^\wedge = E$, and this happens if and only if $(R_{\calS})_x^\wedge$ acts faithfully on $S_\lambda(U)_x^\wedge$ by part~\ref{heckedef:free} of~\ref{thm:heckedef}. 
 

Our choice of lift for the universal type $\calS$ deformation gives a $\CNL_{\calO}$-morphism $R_{\calS}^{\loc} \rightarrow R_{\calS}$. 
We denote again by $x$ the induced $\CNL_{\calO}$-morphisms 
	\[ x : R_{\calS}^{\loc}[1/p] \rightarrow R_{\calS}[1/p] \xrightarrow{x} E \quad \text{and} \quad
	x : R_{\tilde{v}}[1/p] \rightarrow R_{\calS}[1/p] \xrightarrow{x} E \quad \text{for each } \tilde{v} \in \tildeS.\]
Note that for each $\tilde{v} \in \tilde{S}$, the morphism $x : R_{\tilde{v}}[1/p] \rightarrow E$ corresponds to the local representation $r\rest_{G_{\tilde{v}}}$. 
Since $\Pi$ is a cuspidal representation of $\GL_d(\A_F)$, it is generic, 
and its local factors $\Pi_w$ at all finite places are generic. 
Then \ref{thm:automorphicWD} implies that $\WD(r\rest_{G_w}) \cong \iota^{-1} \rec_{F_w}^T(\Pi_w)$ is generic for all finite places $w$. 
Then, by \ref{thm:localdefl}, \ref{thm:pdefring}, and \ref{thm:smoothlp},
	\begin{itemize}
	\item if $\tilde{v} \in \tildeS \smallsetminus \tildeS_p$, then $(R_{\tilde{v}})_x^\wedge$ is formally smooth over $E$ of dimension $d^2$;
	\item if $\tilde{v}\in \tildeS_p$, then $(R_{\tilde{v}})_x^\wedge$ is formally smooth over $E$ of dimension $d^2 + \frac{d(d-1)}{2}[F_{\tilde{v}}:\Q_p]$ (here we used that the $p$-adic Hodge type of $r\rest_{G_{\tilde{v}}} = r_x\rest_{G_{\tilde{v}}}$ is regular by part~\ref{galrepunit:pss} of \ref{thm:galrepunit}).
	\end{itemize}
Then the proof of \cite{KisinFinFlat}*{Lemma~3.4.12} shows
	\begin{enumerate}[resume,label=(\roman*)]
	\item\label{ass:locsmooth} $(R_{\calS}^{\loc})_x$ is formally smooth over $E$ of dimension $d^2\abs{S} + \frac{d(d-1)}{2}[F^+:\Q]$.
	\end{enumerate}


We now use the argument of \cite{ThorneAdequate}*{Theorem~6.8}. 
There are assumptions there that the local deformation rings at places $\tilde{v} \in \widetilde{S}_p$ are crystalline deformation rings, and that $U_v = G(\calO_{F_v^+})$ for $v\in S_p$, but these assumptions are not used in what we need here. 
Since $\zeta_p \notin F$ and $\rbar|_{G_{F(\zeta_p)}}$ has adequate image, 
the proof of \cite{ThorneAdequate}*{Theorem~6.8} using \cite{Thorne2adic}*{Proposition~7.1} in place of \cite{ThorneAdequate}*{Proposition~4.4}, c.f. \cite{Thorne2adic}*{Proposition~7.2}, shows there are nonnegative integers integers $g$ and $q$, such that letting 
	\begin{itemize}
	\item $S_\infty = \calO[[Z_1,\ldots,Z_{d^2\abs{S}},Y_1,\ldots,Y_q]]$;
	\item $\fraka = (Z_1,\ldots,Z_{d^2\abs{S}},Y_1,\ldots,Y_q) \subseteq S_\infty$;
	\item $R_\infty = R_{\calS}^{\loc}[[X_1,\ldots,X_g]]$;
	\end{itemize}
there is an $R_\infty$-module $M_\infty$ with a commuting action of $S_\infty$ such that the following hold.
	\begin{enumerate}[resume,label=(\roman*)]
	\item\label{ass:Rinfty} There is a local $\calO$-algebra morphism $S_\infty \rightarrow R_\infty$ and a surjection $R_\infty \rightarrow R_{\calS}$ of $R_{\calS}^{\loc}$-algebras with $\fraka R_\infty$ in its kernel.
	\item There is a surjection $M_\infty \rightarrow S_\lambda(U)_{\frakm}$ of $R_\infty$ modules with kernel $\fraka M_\infty$, where the $R_\infty$-module structure of $S_\lambda(U)_{\frakm}$ is via the surjection $R_\infty \rightarrow R_{\calS}$ of \ref{ass:Rinfty}.
	\item $M_\infty$ is a finite free $S_\infty$-module.
	\item\label{ass:Num} $g = q - \frac{d(d-1)}{2}[F^+:\Q]$.
	\end{enumerate}
We note that \cite{ThorneAdequate}*{Theorem~6.8} does not state \ref{ass:Rinfty}, but only that the action of $S_\infty$ on $M_\infty$ (denoted $H_\infty$ there) factors through the action of $R_\infty$ on $M_\infty$. 
However, the patching patching part of the argument in \cite{ThorneAdequate}*{Theorem~6.8} is quoted (in \cite{ThorneAdequate}*{Lemma~6.10}) from \cite{BLGGSatoTate}*{Theorem~3.6.1}, where \ref{ass:Rinfty} is shown (in particular, see  \cite{BLGGSatoTate}*{Sublemma in the proof of Theorem~3.6.1}). 
Also, in \ref{ass:Num} we have used that the $\mu_{\frakm}$ in \cite{ThorneAdequate}*{Theorem~6.8} is $d$ in our case. 

Denote again by $x$ the $E$-algebra morphism 
$x : R_\infty[1/p] \rightarrow R_{\calS}[1/p] \xrightarrow{x} E$ and let $\frakp_\infty = \ker(x) \in \Spec R_\infty$. 
Note that $\fraka R_\infty \subseteq \frakp_\infty$. 
This together with the fact that $M_\infty$ is finite free over $S_\infty$ implies we can find an $M_\infty$-regular sequence of length $d^2\abs{S}+q$ in $\frakp_\infty$ and
	\[ \depth_{(R_\infty)_x^\wedge}(M_\infty)_x^\wedge = \depth_{(R_\infty)_x}(M_\infty)_x \ge
	\depth_{R_\infty}(\frakp_\infty,M_\infty) \ge d^2\abs{S} + q, \]
where the first equality follows from \cite{MatsumuraCRT}*{Theorem~23.3} and the middle inequality is \cite{MatsumuraCRT}*{Exercise~16.5}.
On the other hand, using \ref{ass:locsmooth}, \ref{ass:Rinfty}, and \ref{ass:Num} above, we deduce that $(R_\infty)_x^\wedge \cong (R_{\calS}^{\loc})_x^\wedge[[X_1,\ldots,X_g]]$ is formally smooth over $E$ of dimension
	\[ d^2\abs{S} + \frac{d(d-1)}{2}[F^+:\Q] + g = d^2\abs{S} + q. \]
Since $(R_\infty)_x^\wedge$ is regular, a theorem of Serre \cite{MatsumuraCRT}*{Theorem~19.2} and the Auslander--Buchsbaum formula \cite{MatsumuraCRT}*{Theorem~19.1} imply that $(M_\infty)_x^\wedge$ is a finite free $(R_\infty)_x^\wedge$-module. 
Then $(M_\infty)_x^\wedge/\fraka$ is a finite free $(R_\infty)_x^\wedge/\fraka$-module. 
But $(M_\infty)_x^\wedge/\fraka \cong S_\lambda(U)_x^\wedge$ and the action of $(R_\infty)_x^\wedge/\fraka$ on it factors through $(R_{\calS})_x^\wedge$. 
This shows $(R_{\calS})_x^\wedge$ acts faithfully on $S_\lambda(U)_x^\wedge$, which completes the proof of the theorem. 
\end{proof}

\subsection{The proofs of Theorems \ref{thm:BKA}, \ref{thm:thmB}, and \ref{thm:thmC}}\label{sec:introthms}
We now show how each of Theorems~\ref{thm:BKA}, \ref{thm:thmB}, and \ref{thm:thmC} from the introduction follow from \ref{thm:BK}. 

\subsubsection{Proof of Theorem \ref{thm:BKA}}\label{sec:CMThm} 
We recall the setup. 
We have a CM field $F$ with maximal totally real subfield $F^+$ and a finite set of finite places $S$ of $F$ containing all places above $p$. 
We fix a choice of complex conjugation $c \in G_{F^+}$. 
We are given a continuous absolutely irreducible representation
\[ \rho : \Gal(F(S)/F) \lra \GL_d(E). \]
We assume there is a continuous totally odd character $\mu : G_{F^+} \rightarrow E^\times$ and an invertible symmetric matrix $P$ such that the pairing $\langle a, b \rangle = {}^t a P^{-1} b$ on $E^d$ is perfect and satisfies 
	\[ \langle \rho(\sigma) a, \rho(c\sigma c) b \rangle = \mu(\sigma) \langle a, b \rangle. \]
Since $\rho$ is absolutely irreducible, $P$ is unique up to scalar. 
The adjoint representation of $\Gal(F(S)/F)$ on $\ad(\rho) = \frakgl_d(E)$ extends to an action of $\Gal(F(S)/F^+)$ by letting $c$ act by $X \mapsto -P\transp X P^{-1}$, and this is independent of the choice of $c$ and of $P$. 

By choosing a $\Gal(F(S)/F)$-stable $\calO$-lattice in $V_{\rho}$, we may assume that $\rho$ takes values in $\GL_d(\calO)$. 
The semisimplification of its reduction modulo the maximal ideal of $\calO$ does not depend on the choice of lattice, and we denote it by $\rhobar : \Gal(F(S)/F) \rightarrow \GL_d(\F)$.
Assuming there is a finite extension $L/F$ of CM fields, a regular algebraic polarizable cuspidal automorphic representation $\Pi$ of $\GL_d(\A_{L})$, and an isomorphism $\iota : \Qbar_p \xrightarrow{\sim} \C$ such that:
	\begin{enumerate}[label=(\alph*)]
	\item\label{BKA:autagain} $\rho|_{G_{L}} \otimes \Qbar_p \cong \rho_{\Pi,\iota}$;
	\item\label{BKA:adequateagain} $\zeta_p \notin L$ and $\rhobar(G_{L(\zeta_p)})$ is adequate;
	\end{enumerate}
we want to show $H_g^1(F(S)/F^+,\ad(\rho)) = 0$.

By \ref{thm:Gdhoms}, we can define a continuous homomorphism $r : \Gal(F(S)/F^+) \rightarrow \calG_d(E)$ such that $r\rest_{G_F} = \rho$ and $\nu\circ r = \mu$, and there is an isomorphism $\ad(r) \cong \ad(\rho)$ of $\Gal(F(S)/F^+)$-representations. 
We may also assume $r$ takes values in $\calG_d(\calO)$ and letting $\rbar = r \otimes_{\calO} \F$, that $\rbar \rest_{G_F} = \rhobar$. 
Theorem~\ref{thm:BKA} then follows from \ref{thm:BK}. 
\hfill \qed

\subsubsection{Proof of Theorem \ref{thm:thmB}}\label{sec:RealThm}
We have a totally real field $F^+$ and a finite set of finite places $S$ of $F^+$ containing all places above $p$.
Let $F^+(S)$ be the maximal extension of $F^+$ unramified outside $S$ and all places above $\infty$.  
Let
	\[ \rho : \Gal(F^+(S)/F^+) \lra \GL_d(E) \]
be a continuous, absolutely irreducible representation, and let $V_\rho$ denote the representation space of $\rho$. 
We assume that $\rho$ satisfies one of the following:
\begin{enumerate}
	\item[(GO)]\label{Bagain:orthogonal} $\rho$ factors through a map $\Gal(F^+(S)/F^+) \rightarrow \GO_d(E)$, that we again denote by $\rho$, with totally even multiplier character $\mu$;
	\item[(GSp)]\label{Bagain:symplectic} $d$ is even and $\rho$ factors through a map $\Gal(F^+(S)/F^+) \rightarrow \GSp_d(E)$, that we again denote by $\rho$, with totally odd multiplier character $\mu$.
\end{enumerate}
We will refer to the first as the GO-case, and the second as the GSp-case. 

If we are in the GO-case, then we let $\ad(\rho)$ and $\ad^0(\rho)$ denote the Lie algebra $\mathfrak{go}_d(E)$ of $\GO_d(E)$ and sub-Lie algebra $\mathfrak{so}_d(E)$, respectively, with the adjoint action $\ad\circ\rho$ of $\Gal(F^+(S)/F^+)$. 
If we are in the GSp-case, then we let $\ad(\rho)$  and $\ad^0(\rho)$ denote the Lie algebra $\mathfrak{gsp}_d(E)$ of $\GSp_d(E)$ and sub-Lie algebra $\mathfrak{sp}_d(E)$, respectively, with the adjoint action $\ad\circ \rho$ of $\Gal(F^+(S)/F^+)$.

There is a splitting $\mathfrak{go}_d(E) = \mathfrak{so}_d(E) \oplus E$, resp. $\mathfrak{gsp}_d(E) = \mathfrak{sp}_d(E) \oplus E$, with $E$ the Lie algebra of the centre of $\GO_d(E)$, resp. $\GSp_d(E)$, that is stable under the adjoint action. 
We thus get a $\Gal(F^+(S)/F^+)$-equivariant splitting $\ad(\rho) = \ad^0(\rho) \oplus E$, where $E$ has the trivial $\Gal(F^+(S)/F^+)$-action. 
This gives a decomposition 
	\[ H_g^1(F^+(S)/F^+,\ad(\rho)) = H_g^1(F^+(S)/F^+,\ad^0(\rho)) \oplus H_g^1(F^+(S)/F^+,E). \]
For each $v|p$ in $F^+$, the local group $H_g^1(F_v^+,E) = \ker(H^1(F_v^+,E) \rightarrow H^1(F_v^+,B_{\dR}\otimes_{\Q_p}E))$ is the one dimensional $E$-subspace of $\Hom(G_v,E)$ corresponding to the unramified extension \cite{BlochKato}*{Example~3.9}. 
By class field theory, $H_g^1(F^+(S)/F^+,E) = 0$. 
So, we want to show
\begin{enumerate}
\item\label{thing:H1g} $H_g^1(F^+(S)/F^+,\ad^0(\rho)) = 0$;
\item\label{thing:H2} $H^2(F^+(S)/F^+,\ad^0(\rho)) = 0$;
\item\label{thing:H1} for each $v|p$ in $F^+$, the following natural map is an isomorphism
	\[ H^1(F^+(S)/F^+,\ad^0(\rho)) \lra \prod_{v|p} H^1(F_v^+,\ad^0(\rho))/H_g^1(F_v^+,\ad^0(\rho)). \]
\end{enumerate}
The adjoint action of $G_{F^+}$ on $\mathfrak{gl}_d(E) \cong \Hom_E(V_{\rho},V_{\rho})$ decomposes as
	\begin{equation}\label{eqn:gldecomp} 
	\mathfrak{gl}_d(E) \cong \Hom_E(V_{\rho},V_{\rho}) \cong V_{\rho} \otimes V_{\rho}^\vee 
	\cong V_{\rho} \otimes V_{\rho} \otimes \mu^{-1} \cong \Sym^2 V_{\rho}\otimes \mu^{-1} \oplus \textstyle{\bigwedge^2} V_{\rho} \otimes \mu^{-1}.
	\end{equation}
Under this decomposition, $\ad^0(\rho)$ is identified with $\textstyle{\bigwedge^2} V_{\rho}\otimes \mu^{-1}$ in the GO-case, and with $\Sym^2 V_{\rho}\otimes \mu^{-1}$ in the GSp-case.
Theorem~\ref{thm:thmB} then follows from Theorem~\ref{thm:symext} below, noting that we may choose $F$ in the statement of Theorem~\ref{thm:symext} to be any quadratic CM extension of $F^+$ not contained in the subfield of $\overline{F}$ fixed by $\rhobar|_{G_{L^+(\zeta_p)}}$.

\begin{thm}\label{thm:symext}
Let $F/F^+$ be a quadratic CM extension, and let $L^+/F^+$ be a finite totally real extension. 
Assume there is a regular algebraic polarizable cuspidal automorphic representation $\pi$ of $\GL_d(\A_{L^+})$, and an isomorphism $\iota : \Qbar_p \xrightarrow{\sim}\C$ such that
	\begin{enumerate}[label=(\alph*)]
	\item\label{symextass:aut} $\rho|_{G_{L^+}}\otimes \Qbar_p \cong \rho_{\pi,\iota}$;
	\item\label{symextass:ad} $\zeta_p \notin L := FL^+$ and $\rhobar(G_{L(\zeta_p)})$ is adequate;
	\end{enumerate} 
If we are in the GO-case, let 
	\[ W: = \Sym^2 V_\rho \otimes \mu^{-1}\delta_{F/F^+} \quad \text{or} \quad W:=\textstyle{\bigwedge^2} V_{\rho} \otimes \mu^{-1}.\]
If we are in the GSp-case, let 
	\[ W:=\Sym^2 V_\rho \otimes \mu^{-1} \quad \text{or} \quad W:= \textstyle{\bigwedge^2} V_{\rho} \otimes \mu^{-1}\delta_{F/F^+}.\]
Then the following hold.
\begin{enumerate}
\item\label{symext:BK} $H_g^1(F^+(S)/F^+,W) = 0$.
\item\label{symext:H2} $H^2(F^+(S)/F^+,W) = 0$.
\item\label{symext:H1} The natural map
	\[ H^1(F^+(S)/F^+,W) \lra \prod_{v|p} H^1(F_v^+,W) / H_g^1(F_v^+,W) \]
is an isomorphism.
\end{enumerate}
\end{thm}

Before proceeding with the proof, we note that results related to Theorem~\ref{thm:symext}, as well as Theorem~\ref{thm:symmod} below, were obtained by Chenevier \cite{ChenevierFern}*{Theorems~6.11 and 6.12}.

\begin{proof}
The closed subgroup $\Gal(F(S)/F^+(S))$ of $\Gal(F(S)/F^+)$ is either trivial or has order $2$, hence $H^i(F(S)/F^+(S),W) = 0$.  
Inflation-restriction then gives $H^i(F^+(S)/F^+,W) \cong H^i(F(S)/F^+,W)$ for all $i$, and we can replace $F^+(S)/F^+$ with $F(S)/F^+$ in each of \ref{symext:BK}, \ref{symext:H2}, and \ref{symext:H1} above. 

Let $\Pi$ be the base change, using \cite{ArthurClozel}*{Theorem~4.2 of Chapter~3}, of $\pi$ to $\GL_d(\A_L)$. 
The automorphic representation $\Pi$ is regular algebraic and polarizable, and the fact that $\rho|_{G_L}$ is absolutely irreducible implies that it is cuspidal. 
Thus, $\rho|_{G_F}$ satisfies the assumptions of Theorem~\ref{thm:BKA}. 
Fix a choice $c$ of complex conjugation. 
If we are in the GO-case, set $P = \rho(c)$. 
If we are in the GSp-case, set $P = \rho(c)J^{-1}$, where $J$ is the matrix defining the symplectic pairing on $E^d$. 
Then $P$ is an invertible symmetric matrix such that the pairing $\langle a, b \rangle$ on $E^d$ given by ${}^t a P^{-1} b$ satisfies 
	\[ \langle \rho(\sigma) a, \rho(c\sigma c) b \rangle = \mu(\sigma) \langle a, b \rangle, \]
for all $\sigma \in G_F$.
As in \ref{sec:CMThm}, we can extend the adjoint action of $G_F$ on $\Liegl_d(E)$ to an action of $G_{F^+}$ by letting $c$ act as $X \mapsto - P {}^t\! X P^{-1}$, and we let $\ad(\rho|_{G_F})$ denote $\Liegl_d(E)$ with this $G_{F^+}$-action. 

Then \eqref{eqn:gldecomp} gives a decomposition of $G_F$-representations
	\begin{equation}\label{eqn:addecomp} 
	\ad(\rho|_{G_F}) \cong \Sym^2 V_{\rho}\otimes \mu^{-1} \oplus \textstyle{\bigwedge^2} V_{\rho} \otimes \mu^{-1},
	\end{equation}
where 
	\begin{itemize}
	\item in the GO-case, $\Sym^2 V_{\rho}\otimes \mu^{-1}$ is identified with the symmetric matrices, and $\textstyle{\bigwedge^2} V_{\rho}\otimes \mu^{-1}$ with the antisymmetric matrices.
	\item in the GSp-case, $\Sym^2 V_{\rho}\otimes \mu^{-1}$ is identified with the matrices $X$ such that $\transp XJ = -J X$, and $\textstyle{\bigwedge^2} V_{\rho}\otimes \mu^{-1}$ with the matrices $X$ such that $\transp XJ = JX$;
	\end{itemize}
Calculating the action of $c$ on both sides of \eqref{eqn:addecomp}, we find
	\[\ad(\rho|_{G_F}) \cong \Sym^2 V_{\rho}\otimes \mu^{-1}\delta_{F/F^+} \oplus \textstyle{\bigwedge^2} V_{\rho} \otimes \mu^{-1},\]
as $G_{F^+}$-representations in the GO-case, and 
	\[\ad(\rho|_{G_F}) \cong \Sym^2 V_{\rho}\otimes \mu^{-1} \oplus \textstyle{\bigwedge^2} V_{\rho} \otimes \mu^{-1}\delta_{F/F^+},\]
as $G_{F^+}$-representations in the GSp-case.
The theorem now follows from Theorem~\ref{thm:BKA}.
\end{proof}
\subsubsection{Proof of Theorem C}\label{sec:smooth}
We recall the setup.  
Fix a continuous homomorphism 
	\[ \rbar : \Gal(F(S)/F^+) \lra \calG_d(\F) \]
inducing an isomorphism $\Gal(F/F^+) \xrightarrow{\sim} \calG_d(\F)/\calG_d^0(\F)$, as well as a totally odd \mbox{de Rham} character $\mu : \Gal(F(S)/F^+) \rightarrow \calO^\times$ with $\nu \circ \rbar = \mu \bmod {\frakm_{\calO}}$. 
We assume $\rbar\rest_{G_F}$ is absolutely irreducible.  
The ring $R_{\calS}$ in the statement of Theorem~\ref{thm:thmC} is the universal type $\calS$ deformation ring for the global deformation datum
	\[ \calS = (F/F^+,S,\emptyset,\calO,\rbar,\mu,\emptyset). \]
Theorem~\ref{thm:thmC} then follows immediately from \ref{thm:BK} and \ref{thm:globalsmooth}. \hfill\qed

\subsection{Symmetric powers}\label{sec:sympow}

We finish by giving one application of our main result combined with potential automorphy theorems. 

Let $f$ be an elliptic modular newform of weight $k\ge 2$ and level $\Gamma_1(N)$.  
Fix a prime $p$, a sufficiently large finite extension $E/\Q_p$, and an embedding of the coefficient field of $f$ into $E$. 
Let $S$ be a finite set of primes of $\Q$ containing all primes dividing $Np$, and let
	\[ \rho_f : \Gal(\Q(S)/\Q) \lra \GL(V_f) \cong \GL_2(E) \]
be the associated $p$-adic Galois representation. 
We let $\F$ be the residue field of $E$, and let 
	\[ \rhobar_f: \Gal(\Q(S)/\Q) \lra \GL_2(\F) \]
denote the (semisimple) residual representation of $\rho_f$.
For any $n\ge 1$, let
	\[ \Sym^n(\rho_f) : \Gal(\Q(S)/\Q) \lra \GL(\Sym^n V_f) \cong \GL_{n+1}(E) \]
be the representation on the $n$th symmetric power $\Sym^n V_f$ of $V_f$. 
. 

\begin{thm}\label{thm:symmod}
Let $F/\Q$ be any imaginary quadratic field.
Let $n \ge 1$, and set
	\[ W_n := \Sym^{2n} V_f \otimes \det(\rho_f)^{-n}\delta_{F/\Q}^{n+1}. \]
Then $H_g^1(\Q(S)/\Q, W_n) = 0$ under the following assumptions:
	\begin{enumerate}[label=(\alph*)]
	\item\label{symmodass:p} $p > n+4$ and $p \notin \{2n+1,2n+3\}$,
	\item\label{symmodass:im} the image of $\rhobar$ contains $\SL_2(\F_p)$.
	\end{enumerate}
\end{thm}

\begin{proof}
Let $\rho := \Sym^n(\rho_f)$,
and let $\rhobar : \Gal(\Q(S)/\Q) \rightarrow \GL_{n+1}(\F)$ be the (semisimple) residual representation. 
Note that $\rhobar$ is isomorphic to the semisimplification of $\Sym^n(\rhobar_f)$.
Set 
	\[ W:= \Sym^2(\Sym^n V_f) \otimes \det(\rho_f)^{-n}\delta_{F/\Q}^{n+1}. \]
There is a natural symmetric bilinear map $\Sym^n V_f \times \Sym^n V_f \rightarrow \Sym^{2n} V_f$, which induces a surjection
	\[ \Sym^2 (\Sym^n V_f) \lra \Sym^{2n} V_f. \]
This map is $\GL_2(E)$-equivariant and has a $\GL_2(E)$-equivariant splitting. 	
Thus, $W_n$ is a $G_{\Q}$-equivariant direct summand of $W$, so it suffices to show $H_g^1(\Q(S)/\Q,W) = 0$. 
To do this, we apply Theorem~\ref{thm:symext}.

If $n$ is even, then $\rho$ is conjugate to a representation valued in $\GO_{n+1}(E)$, and if $n$ is odd, to one valued in  $\GSp_{n+1}(E)$, and in either case the multiplier character is $\det(\rho_f)^n$. 
Assumption \ref{symmodass:im} implies that $f$ does not have complex multiplication, so \cite{BLGHT}*{Theorem~B} implies there is a totally real Galois extension $L^+/\Q$, a regular algebraic polarizable cuspidal automorphic representation $\pi$ of $\GL_{n+1}(\A_{L^+})$, and an isomorphism $\iota : \Qbar_p \xrightarrow{\sim}\C$ such that $\rho|_{G_{L+}} \otimes \Qbar_p \cong \rho_{\pi,\iota}$. 
Moreover, we may assume that $L^+$ is disjoint from the subfield of $\Qbar$ fixed by $\rhobar_f|_{G_{F(\zeta_p)}}$ (for example, see \cite{BLGGT}*{Theorem~5.4.1}). 
Note that assumption \ref{symmodass:p} implies $p\ge 5$. 
So $\zeta_p \notin F$, and the image of the restriction $\rhobar_f|_{G_{F(\zeta_p)}}$ still contains $\SL_2(\F_p)$. 
By choice of $L^+$, we have $\zeta_p \notin L:= FL^+$ and $\rhobar_f(G_{L(\zeta_p)})$ contains $\SL_2(\F_p)$. 
This implies, in particular, that the image of $G_{L(\zeta_p)}$ under $\Sym^n (\rhobar_f)$ acts absolutely irreducibly and $\rhobar \cong \Sym^n (\rhobar_f)$. 
Under assumption \ref{symmodass:p}, \cite{GHTadequateLD}*{Corollary~1.5} implies $\rhobar(G_{L(\zeta_p)})$ is adequate. 
We've verified the assumptions of Theorem~\ref{thm:symext} for $\rho$, so $H_g^1(\Q(S)/\Q,W) = 0$.
\end{proof}

\end{document}